\documentclass[a4paper,12pt]{amsart}
\usepackage{amsmath,amsthm,amssymb,amsfonts,enumerate,color,hyperref,MnSymbol,wasysym}
\usepackage[pdftex]{graphicx}
\usepackage[bottom=2.5cm, top=2cm]{geometry}
\usepackage{float}
\usepackage{tikz}
\usepackage[makeroom]{cancel}

\theoremstyle{plain}
\newtheorem{theorem}{Theorem}[section]
\newtheorem{lemma}[theorem]{Lemma}
\newtheorem{corollary}[theorem]{Corollary}

\theoremstyle{definition}
\newtheorem{remark}[theorem]{Remark}

\newtheorem{definition}[theorem]{Definition}

\numberwithin{equation}{section}

\def\sgn{\operatorname{sgn}}

\def\Re{\operatorname{Re}}
\def\Im{\operatorname{Im}}
\def\im{\operatorname{Im}\left(\frac{1}{\Phi'}\right)}

\def\re{\operatorname{Re}\left(\frac{1}{\Phi'}\right)}

\def\A{\operatorname{Arg}}

\def\R{\mathbb R}
\def\Z{\mathbb Z}

\def\La{\Lambda}

\def\e{\varepsilon}

\RequirePackage{fix-cm}

\usepackage{times}
\usepackage{comment}
\usepackage[latin1]{inputenc}
\usepackage[greek,english]{babel}
\usepackage{enumerate}
\usepackage{multicol,stackrel}
\usepackage{geometry}
\usepackage[all]{xy}
\usepackage{tikz}
\usepackage{pgfplots}
\pgfplotsset{compat=1.11}
\usepgfplotslibrary{fillbetween}
\usetikzlibrary{patterns}

\newcommand{\fr}[2]{\textstyle{\frac{#1}{#2}}}
\newcommand{\ntmp}{\mathcal M^+} 
\newcommand{\ntmm}{\mathcal M^-} 
\newcommand{\ntmpm}{\mathcal M^{\pm}} 

\newcommand{\abs}[1]{\left|#1\right|}

\newcommand{\op}{{\Omega^+}}
\newcommand{\om}{{\Omega^-}}
\newcommand{\opm}{\Omega^{\pm}}
\newcommand{\phip}{\Phi_+}
\newcommand{\phim}{\Phi_-}
\newcommand{\phipm}{\Phi_\pm}
\newcommand{\phipi}{\Phi_+^{-1}}
\newcommand{\phimi}{\Phi_-^{-1}}
\newcommand{\phipmi}{\Phi_\pm^{-1}}

\newcommand{\hpp}{\R^2_+}
\newcommand{\hpm}{\R^2_-}
\newcommand{\hppm}{\R^2_\pm}

\newcommand{\up}{u^+}
\newcommand{\um}{u^-}
\newcommand{\upm}{u^\pm}
\newcommand{\vp}{v^+}
\newcommand{\vm}{v^-}
\newcommand{\vpm}{v^\pm}

\newcommand{\normal}{{\bf{n}}}

\newcommand{\tphim}{T_{\phim}}
\newcommand{\tphipm}{T_{\phipm}}

\newcommand{\tphipi}{T_{\phipi}}
\newcommand{\tphimi}{T_{\phimi}}

\newcommand{\psii}{\Psi^{-1}}

\newcommand{\tpsi}{T_{\Psi}}
\newcommand{\tpsii}{T_{\Psi^{-1}}}

\newcommand{\lpr}[2]{L^{#1}(\R,#2)}
\newcommand{\lprnw}[1]{L^{#1}(\R)}
\newcommand{\lpl}[2]{L^{#1}(\La,#2)}
\newcommand{\lplnw}[1]{L^{#1}(\La)}

\newcommand{\apr}[1]{A_{#1}(\R)}

\newcommand{\wtil}{\widetilde{w}}
\newcommand{\wtilp}[1]{\widetilde{w}_{#1}}

\newcommand{\otil}{\widetilde{\Omega}}

\newcommand{\Ltil}{\widetilde{\Lambda}}

\newcommand{\kpsi}{k_{\Psi}}
\newcommand{\apsi}{A_{\Psi}}

\newcommand{\cpsi}{C_{\Psi}}

\newcommand{\tp}[2]{P_{#1}{(#2)}}

\newcommand\restr[2]{{
  \left.\kern-\nulldelimiterspace 
  #1 
  \littletaller 
  \right|_{#2} 
  }}
 \newcommand{\littletaller}{\mathchoice{\vphantom{\big|}}{}{}{}}

\begin{document}

\title{Transmission problems for simply connected domains in the complex plane} 
\author {Mar\'{\i}a J.\ Carro, Virginia Naibo and Mar\'{\i}a Soria-Carro} 

\address{Mar\'{\i}a Jes\'us Carro.  Department of  Analysis and Applied Mathematics,              
Universidad Complutense de Madrid, Plaza de las Ciencias 3,
28040 Madrid, Spain. Instituto de Ciencias Matem\'aticas ICMAT, Madrid, Spain}
\email{mjcarro@ucm.es}

\address{Virginia Naibo. Department of Mathematics, Kansas State University.
138 Cardwell Hall, 1228 N. Martin Luther King Jr. Dr., KS  66506, USA.}
\email{vnaibo@ksu.edu}

\address{Mar\'{\i}a Soria-Carro. Department of Mathematics, Rutgers University, 110 Frelinghuysen Rd.
Piscataway, NJ 08854, USA.}
\email{maria.soriacarro@rutgers.edu}


\thanks{First author partially supported by grants PID2020-113048GB-I00   and CEX2019-000904-S funded by MCIN/AEI/10.13039/501100011033,  and Grupo UCM-970966 (Spain).  Second author partially supported by the NSF under grant DMS 2154113 (USA). Third author partially supported by the NSF under grant DMS 2247096 (USA)}

\begin{abstract} We study existence and uniqueness of a transmission problem  in simply connected domains in the plane with data  in weighted Lebesgue spaces by first investigating solvability results of a related novel problem associated to a homeomorphism in the real line and domains given by the upper and lower half planes.  Our techniques are based  on the use of conformal maps and Rellich identities for the Hilbert transform, and are motivated by previous works concerning the Dirichlet, Neumann and Zaremba problems.  

  \end{abstract}

\maketitle
\pagestyle{headings}\pagenumbering{arabic}\thispagestyle{plain}

\section{Introduction and main results}\label{sec:intro}

Transmission problems model diffusion processes with discontinuities across interfaces.  They arise in areas such as electrodynamics of solid media and solid mechanics as well as in the study of vibrating folded membranes and composite plates, among other settings. The pioneering  work \cite{picone} related to partial differential equations in classical elasticity along with subsequent investigations in \cite{MR0094579, MR0089978, MR0082607,MR0131063} propelled an active field of research.  Well-posedness, regularity theory, interfaces and numerous other aspects have been extensively studied ever since; see for instance the articles \cite{MR1092919, MR2091359, MR1149120, MR1770682} and the more recent works  \cite{MR4228861,   MR3356996,  MR3586566, MR4658348}, as well as references therein.  We refer the reader to the monograph  \cite{MR2676605} for a comprehensive treatment of the theory of transmission problems.

\medskip

  In this article, we study existence and uniqueness of solutions of a  transmission problem in the plane with data in weighted Lebesgue spaces, which we next describe.  Let $\La$ be an unbounded rectifiable Jordan curve that divides  the complex plane in two simply connected domains, $\Omega^+$ (upper  graph domain) and $\Omega^{-}$ (lower  graph domain), and consider the following transmission problem (see Figure~\ref{fig:gralsetting}):
\begin{align}\label{eq:TP}
\begin{cases}
\Delta \vpm=0 &  \hbox{in}~\opm,\\ 
\vp=\vm  & \hbox{on}~\La,\\
\partial_{\normal}\vp- \mu \,\partial_\normal \vm =g & \hbox{on}~\La.
\end{cases}
\end{align}
Here, $\Delta$ is the Laplacian in $\mathbb{R}^2$, $\normal$ is the inward unit normal to $\Omega^+$, $\partial_\normal \vpm$ denotes the normal derivative of $\vpm$,  and $\mu\in \R$ is a fixed parameter. The datum $g$ belongs to  some weighted Lebesgue space  on $\Lambda$  and equalities on the boundary are  interpreted almost everywhere with respect to arc length. Estimates in weighted Lebesgue spaces for appropriate maximal operators (for instance, the non-tangential maximal operator when the domains have the cone property) of the gradients of the solutions  are also required   (see Definition~\ref{def:TPsol} for details). We note that when $\mu=0$, \eqref{eq:TP} reduces to a Neumann problem for $\vp$ and to a Dirichlet problem for $\vm.$

 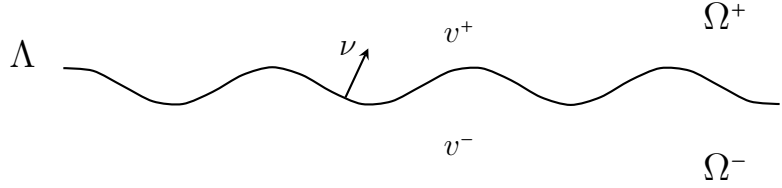
\begin{figure}[htbp]
\begin{center}
\begin{tikzpicture}[scale=0.5]

\draw [smooth, domain=-3*pi:3*pi,  thick] plot(\x, {.5*sin(1.2*\x r)});
\draw [line width=0.8pt,-stealth] (-2,-0.3)--(-2+0.6,-0.3+1.3) node[left]{$\nu$};
\node[below] at (8,2.5) {\large$\Omega^+$};
\node[below] at (8,-1.5) {\large$\Omega^-$};
\node [below] at (-10.5,1.5)  {\large$\Lambda$};
\node [below] at (1,2) {{$v^+$}};
\node [below] at (1,-1) {{$v^-$}};
\end{tikzpicture}
\caption{The transmission problem \eqref{eq:TP}}
\label{fig:gralsetting}
\end{center}
\end{figure}

The solvability of the transmission problem \eqref{eq:TP} with data in unweighted $L^p$ spaces was studied  in \cite{MR2091359} for the case when $\Lambda$ is the graph of a Lipschitz function. It holds that if $\mu \neq 1$, then there exists $\e=\e(\Lambda,\mu)>0$ such that the transmission problem has a unique solution (up to constants) for all $g\in \lplnw{p}$ provided that $1<p<2+\e$; moreover, there are integral representation formulas for the solution in terms of harmonic layer potentials. When $\mu=1$, problem \eqref{eq:TP} is well-posed for any datum in $\lplnw{p}$ and $1<p<\infty.$ It was proved in \cite{MR2299477} that the range $1<p\le 2$ is sharp in the sense that for any $p>2$ there exist  Lipschitz domains $\opm$ and $0<\mu<1$ such that \eqref{eq:TP} is not well-posed in $\lplnw{p}$. 
\medskip

In this article, we investigate new aspects associated to  the solvability of the transmission problem \eqref{eq:TP} in weighted Lebesgue spaces, and specially in the case $p=2,$  by first studying a related problem corresponding to $\opm=\hppm,$ which is interesting in its own right and for which we obtain novel results.  Our techniques are based  on the use of conformal maps and Rellich identities for the Hilbert Transform, and are  motivated by the works \cite{MR864372, MR556889, MR545265}, as well as the articles  \cite{MR3800109, MR4542711,  CLN2023} dealing with the Dirichlet, Neumann and Zaremba problems. 

Since $\opm$ are simply connected, they are conformally equivalent to $\hppm.$
 Let $\phipm:\hppm\to \opm$  be conformal maps such that $\phipm(\infty)=\infty.$ 
 Then $\phipm$ extend as  homeomorphisms from $\overline{\hppm}$ onto $\overline{\opm}$ (Carath\'eodory's theorem) and $\phipm'(x)$ exist  and are not zero for almost every $x\in \R$ (see Section~\ref{sec:phi}).   For an almost everywhere differentiable function $\beta$ defined on $\R$ with values in the domain of a complex-valued function $g$, set $$T_\beta g=|\beta'|\,(g\circ \beta).$$   Consider solutions $\vpm$  of  the transmission problem \eqref{eq:TP} and define $\upm=\vpm\circ \phipm.$ 
  \begin{figure}[htbp]
\begin{center}
\begin{tikzpicture}[scale=.5]
\draw [smooth, domain=0:3*pi,  thick] plot(\x, {.5*sin(2*\x r)});
\node[below] at (8,2.5) {\large$\Omega^+$};
\node[below] at (8,-1.5) {\large$\Omega^-$};
\node [below] at (-1,1.5)  {\large$\Lambda$};
\node [below] at (3,2) {{$v^+$}};
\node [below] at (3,-1) {{$v^-$}};
\draw [->, thick] (15,0) -- (23,0);
\draw [->, thick] (19,-3.5) -- (19,3.5);
\node [below] at (22, 2) {$u^+=v^+\circ \Phi_+$};
\node [above] at (22, -2) {$u^-=v^-\circ \Phi_-$};
\node[below] at (16,3) {\large$\R^2_+$};
\node[below] at (16,-1.5) {\large$\R^2_-$};
\draw[thick,<-] (10,2) arc (130:40:3);
\draw[thick,<-] (10,-2) arc (230:320:3);
\node [above] at (12.2,3) {{\large${\color{black}\Phi_+}$}};
\node [below] at (12.2,-3) {{\large${\color{black}\Phi_-}$}};
\end{tikzpicture}
\caption{The conformal maps $\phipm$ }
\label{fig:tranportealplano}
\end{center}
\end{figure}
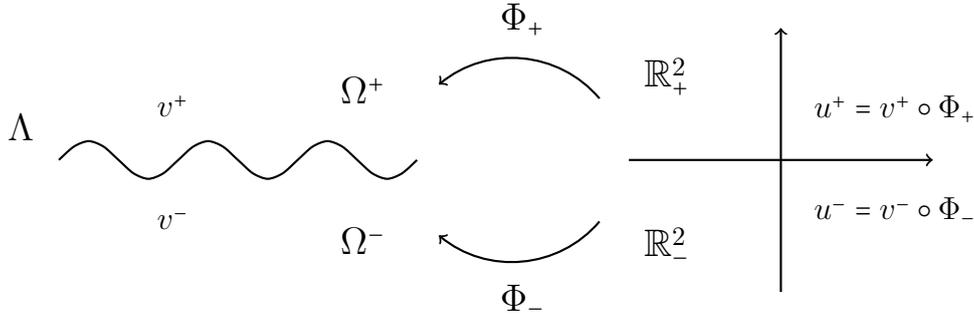

Then $\upm$ are harmonic in $\hppm$ and $\up\circ \Psi=\um$ on $\R$, where $\Psi=\phipi\circ\phim.$  Using that $\partial_y \upm=\tphipm(\partial_\normal \vpm)$  on $\R$ (see Section~\ref{sec:solTP}), we have 
\begin{align*}
\tphipi(\partial_y \up)-\mu\, \tphimi(\partial_y \um)=\partial_{\normal}\vp- \mu \,\partial_\normal \vm =g\quad \text{on } \La.
\end{align*}
This leads to 
\begin{align*}
\tphim\tphipi(\partial_y \up)-\mu\, \partial_y \um=\tphim g\quad \text{on } \R,
\end{align*}
and noting that $\tpsi=\tphim\circ\tphipi$, we obtain  
\begin{align*}
\tpsi(\partial_y \up)-\mu\, \partial_y \um=\tphim g \quad \text{on } \R.
\end{align*}
We then have that $\upm$ satisfy 
$\Delta \upm=0$ in $\hppm,$  $\up\circ\Psi=\um$ on $\R$, and 
$\tpsi(\partial_y \up)-\mu\, \partial_y \um=\tphim g$ on $\R$.
More generally, we will consider a homeomorphism $\Psi:\R\to \R$ and study the solvability of the following problem with  $f\in \lpr{p}{w}$, where $w$ is a weight in the Muckenhoupt class $\apr{p}:$

\begin{align}\label{eq:TPplane}
\begin{cases}
\Delta \upm=0 &  \hbox{in}~\hppm,\\ 
\up\circ\Psi=\um  & \hbox{on}~\R,\\
\tpsi(\partial_y \up)-\mu\, \partial_y \um=f & \hbox{on}~\R.
\end{cases}
\end{align}
We note that when $\Psi(x)=x$, \eqref{eq:TPplane} corresponds to the transmission problem \eqref{eq:TP} with $\opm=\hppm$, which can be solved through the Neumann problem and the reflection principle for harmonic functions. We will denote by $\tp{\Psi}{\mu}$ the transmission problem \eqref{eq:TPplane}.  As in the case of \eqref{eq:TP}, $\tp{\Psi}{0}$ reduces to a Neumann problem for $\up$ and to a Dirichlet problem for $\um,$ whose solvability  with data in weighted Lebesgue and Lorentz spaces was studied in \cite{MR4542711,MR3800109}; therefore, we will always assume that $\mu\neq 0.$

By studying  existence and uniqueness of solutions of \eqref{eq:TPplane}, we obtain results on existence and uniqueness of solutions of \eqref{eq:TP}, and conversely:

\begin{theorem}\label{thm:tplwgral} Let $1<p<\infty,$ $\mu\neq 0$, $\nu$ be a weight on $\Lambda$, and $ w=|\phim'|^{1-p}\,(\nu\circ \phim)$; assume    $\Psi=\phipi\circ \phim$ is locally absolutely continuous.   Then the transmission problem~\eqref{eq:TP} is uniquely solvable in $\lpl{p}{\nu}$ if and only if  $\tp{\Psi}{\mu}$ is uniquely solvable in $\lpr{p}{w}.$ 
\end{theorem}

We refer the reader to Definitions~\ref{def:pmu} and \ref{def:TPsol} regarding the concepts of  unique solvability for $\tp{\Psi}{\mu}$ and \eqref{eq:TP}, respectively.  As  corollaries of Theorem~\ref{thm:tplwgral}  we obtain unique solvability  of  \eqref{eq:TP} in $\lplnw{2}$, $\lpl{2}{|(\phipi)'|^{-1}}$ and $\lpl{2}{|(\phimi)'|^{-1}}$ (see Section~\ref{sec:solTP}). In particular, we  recover results in \cite[Theorem 1.1]{MR2091359} for $p=2$ and $n=2$ in the case when $\Lambda$ is the graph of a Lipschitz function. In the following, if $\Lambda$ is the graph of a Lipschitz function, we will refer to $\op$ as an upper  Lipschitz graph domain and to $\om$ as a lower  Lipschitz graph domain.
 Our results are true for a general unbounded rectifiable Jordan curve $\Lambda$; for instance, we have the following corollary when $\Lambda$ is the hyperbola $y=1/x:$

\begin{corollary}\label{coro:hyper2} Let $\opm$ be the upper and lower  graph domains associated to the hyperbola $y=1/x,$ $x>0.$ Then the  transmission problem \eqref{eq:TP} is uniquely solvable in $\lpl{2}{|(\phimi)'|^{-1}}$ for  $\mu\neq 0$ such that $|\mu|<\frac{-\sqrt{3}+\sqrt{3+2^{1/3}}}{2}\approx 0.165953$ or  $|\mu|>\frac{2}{-\sqrt{3}+\sqrt{3+2^{1/3}}}\approx6.02579.$ 
\end{corollary}

We start our study of $\tp{\Psi}{\mu}$ by proving necessary and sufficient conditions for the solvability and unique solvability of $\tp{\Psi}{\mu}$ in  $\lpr{p}{w}$ in terms of the surjectivity and invertibility, respectively, of the operator $H\tpsi +\mu \,\tpsi H$, where $H$ is the Hilbert transform. More precisely, we prove the following result (see Section~\ref{sec:tpplanegral}):
\begin{theorem}\label{thm:TPplanesol} Let $\mu\neq 0$, $1<p<\infty,$ and $\Psi:\R\to\R$ be a locally absolutely continuous homeomorphism with $\Psi'>0$ almost everywhere; assume $w\in \apr{p}$ is such that $\wtilp{p}=|(\Psi^{-1})'|^{1-p} (w\circ\Psi^{-1})\in \apr{p}.$ It holds that 
\begin{enumerate}[(a)]
\item\label{item:TPplanesol1} $\tp{\Psi}{\mu}$ is solvable in $\lpr{p}{w}$ if and only if the operator $H\tpsi +\mu \,\tpsi H:\lpr{p}{\wtilp{p}}\to \lpr{p}{w}$ is surjective;
\item\label{item:TPplanesol2}   $\tp{\Psi}{\mu}$ is uniquely solvable in $\lpr{p}{w}$ if and only if the operator $H\tpsi +\mu \,\tpsi H:\lpr{p}{\wtilp{p}}\to \lpr{p}{w}$ is invertible.
\end{enumerate}
\end{theorem}  

Theorem~\ref{thm:TPplanesol} is used to prove   solvability results for $\tp{\Psi}{\mu}$ in weighted $L^2$ spaces. For instance,  we give sufficient conditions on the homeomorphism $\Psi$ for the solvability of $\tp{\Psi}{\mu}$  in $\lpr{2}{\frac{1}{\Psi'}}$ (see  Section~\ref{sec:solvL2} for other related results):  

\begin{theorem}\label{thm:hyper} Let $\Psi:\R\to \R$ be a locally absolutely continuous homeomorphism such that $\Psi'>0$ almost everywhere  and satisfies 
\begin{equation}\label{eq:hyper0}
\frac 1{\Psi'}=\re
\end{equation}
for some  conformal map $\Phi$ from $\hpp$ onto an upper  Lipschitz graph domain. Define 
$$
\kpsi:=\left\| \Psi'\,{\im}\right\|_{\lprnw{\infty}}.
$$
Then  for every $0<|\mu|<1$ such that
\begin{equation}\label{eq:muk}
 1- \mu^2 - 2\,\kpsi \,|\mu|  >0,
\end{equation}
the transmission problems $\tp{\Psi}{\mu}$  and $\tp{\Psi}{1/\mu}$ are uniquely solvable in $\lpr{2}{\frac{1}{\Psi'}}.$ \end{theorem}

We also present several examples where the conformal map $\Phi$ in Theorem~\ref{thm:hyper} is associated  to domains such as an infinite staircase, symmetric cones and hyperbolas, as well as to the Helson-Szeg\"o representation of $A_2$ weights.

Returning to the motivation problem described in this section, we consider homeomorphisms $\Psi:\R\to \R$ of the form $\Psi=\phip^{-1}\circ \phim$  and prove the following result:

\begin{theorem}\label{thm:welding} Let $\Psi=\phip^{-1}\circ \phim$ with $\opm$ upper and lower Lipschitz graph domains. Then  $\tp{\Psi}{\mu}$ is solvable in $\lpr{2}{|\phim'|^{-1}}$ for $\mu>0.$ 
\end{theorem}

Homeomorphisms of the form $\Psi=\phip^{-1}\circ \phim$ are referred to as conformal weldings and have been extensively studied. More generally, a homeomorphism $\psi: \mathbb{R} \rightarrow \mathbb{R}$ is a conformal welding if there exist a Jordan curve $\Gamma$ in $\overline{\mathbb{C}}$ with complementary domains $\Omega_1$ and $\Omega_2$, and conformal mappings $\phi_1: \hpp \rightarrow \Omega_1$ and $\phi_2: \hpm \rightarrow \Omega_2$ such that $\psi=\phi_1^{-1} \circ \phi_2.$ A sufficient condition for a homeomorphism $\psi$ to be a conformal welding is quasi-symmetry: if $\psi$ is an increasing function so that there is a constant $M>0$ such that 
$$\frac{1}{M}\leq \frac{\psi(x+t)-\psi(x)}{\psi(x)-\psi(x-t)} \leq M,\quad \forall x\in\R, t>0,$$
then $\psi$ is a conformal welding. This condition is equivalent to saying that the push forward of the Lebesgue measure under $\psi$ is a doubling measure. We refer the reader  to \cite{MR0200442, MR0086869,MR1081289, MR2373370, MR4480881,MR2041702} for more information regarding conformal weldings.

The organization of the article is as follows. In Section~\ref{sec:prelim}, we present some notation, definitions and preliminary results that will be used throughout the article. These include the definitions of  classes of weights and several associated properties, the precise definitions of solvability for \eqref{eq:TP} and  $\tp{\Psi}{\mu},$ and statements on the Neumann problem in the upper half plane proved in \cite{MR4542711} along with new results regarding its uniqueness of solutions and boundary data (Theorem~\ref{thm:uniqueN} and Lemma~\ref{lem:localint}). Section~\ref{sec:tpplanegral} contains the proof of Theorem~\ref{thm:TPplanesol}  and the statement and proof of Corollary~\ref{coro:sym}, which  gives symmetric properties associated to the solvability of $\tp{\Psi}{\mu}$. The proofs of Theorem~\ref{thm:hyper} and Theorem~\ref{thm:welding} are presented in Section~\ref{sec:solvL2} along with the statements and proofs of related results (Theorem~\ref{thm:hypergral} and Theorem~\ref{thm:hilb}) and examples.  A main tool in the proofs of  the results in Section~\ref{sec:solvL2} is the use of a type of Rellich identity for the Hilbert Transform proved in  \cite{MR4674966}, which we describe in Section~\ref{sec:rellich}. Section~\ref{sec:solTP} contains the proof of Theorem~\ref{thm:tplwgral} and the statements of results on the solvability of  \eqref{eq:TP} in $\lplnw{2}$, $\lpl{2}{|(\phipi)'|^{-1}}$ and $\lpl{2}{|(\phimi)'|^{-1}}$ that follow from Theorem~\ref{thm:tplwgral} and the theorems in Section~\ref{sec:solvL2}. An application of Theorem~\ref{thm:hypergral}  related to the hyperbola $y=1/x,$ $x>0,$ and the proof of Corollary~\ref{coro:hyper2} are presented on Section~\ref{sec:hyperbola}

\section{Notation and preliminaries}\label{sec:prelim}

In this section we present some notation, definitions and preliminary results that will be used throughout the article. 

The notation $A\lesssim B$ means that there is a constant $c>0$ such that $A\le c\, B$; $c$ may depend on some of the parameters used but not on the functions or variables involved. We will use $A\approx B$ if $A\lesssim B$ and $B\lesssim A$.

\subsection{Weights}\label{sec:prelim:fsw}  Consider a weight $w$ on $\R$  (i.e. a non-negative locally integrable function defined in $\R).$  For $1\le p\le\infty$, the space $\lpr{p}{w}$ is the class of  measurable functions  $f:\R\to\mathbb{C}$ such that 
\begin{equation*}
 \|f\|_{\lpr{p}{w}}=\left(\int_\R |f(x)|^pw(x)\,dx\right)^{\frac{1}{p}}<\infty,
 \end{equation*}
 with the corresponding changes for $p=\infty.$ 
When $w\equiv 1$,  we use the notation $L^p(\R)$ instead of $\lpr{p}{w}.$

If $1<p<\infty$, the Muckenhoupt class $\apr{p}$  is given by the weights $w$ such that   
\begin{equation*}
[w]_p=\sup _{I\subset \R}\, \left(\frac{1}{|I|}\int_{I} w(x) d x\right)\left(\frac{1}{|I|}\int_{I} w(x)^{1-p^{\prime}}d x\right)^{p-1}<\infty,\label{eq:Ap}
\end{equation*}
where the supremum is taken over all intervals $I\subset \R$ and $p'$ is the conjugate exponent of $p$ (i.e. $1/p+1/p'=1$).  
We recall that  $w\in \apr{p}$ if and only if $w^{1-p'}\in \apr{p'}$,  $\apr{p}\subset \apr{q}$ if $p<q$ and, if $w\in \apr{p}$ then $w\in \apr{p-\varepsilon}$ for some $\varepsilon>0.$

The Hilbert transform $H$ is defined by 
\begin{equation*}
H f (x)=\frac{1}{\pi}\lim_{\varepsilon\to 0} \int_{|t|>\varepsilon} \frac{f(x-t)}{t}\,dt.
\end{equation*}
The Muckenhoupt classes characterize boundedness properties of $H$ in the sense that $H$ is bounded on $\lpr{p}{w}$ if and only if $w\in \apr{p}$ (see Hunt--Muckenhoupt-Wheeden~\cite{MR312139}).

\subsubsection{Homeomorphisms and weights}
Let $\Psi:\R\to\R$ be a homeomorphism and $w$ a weight in $\R;$ define  
\begin{equation}\label{eq:wwtil}
\wtilp{p}=|(\Psi^{-1})'|^{1-p} (w\circ\Psi^{-1})=\tpsii(w|\Psi'|^p).
\end{equation}
Then $\tpsi:\lpr{p}{\wtilp{p}}\to \lpr{p}{w}$ and  we have
\begin{align*}
\|\tpsi h\|_{\lpr{p}{w}}^p&=\int_\R |h(\Psi(x))|^p|\Psi'(x)|^pw(x)\,dx=\int_\R |h(y)|^p |\Psi'(\Psi^{-1}(y))|^{p-1} w(\Psi^{-1}(y))\,dy\\
&=\int _\R |h(y)|^p |(\Psi^{-1})'(y)|^{1-p} w(\Psi^{-1}(y))\,dy=\|h\|_{\lpr{p}{\wtilp{p}}}^p.
\end{align*}

\begin{remark} When $p=2$, we will use the notation $\wtil$ instead of $\wtilp{2}.$
\end{remark}

The following remark states necessary and sufficient conditions for $\wtilp{p}$ to be in $\apr{p}$ when $w\in \apr{p}.$

\begin{remark}\label{re:wtildeAp} Let $\Psi:\R\to\R$ be a locally absolutely continuous homeomorphism.
If $w\in \apr{p}$, then 
\begin{equation*}
\wtilp{p}\in \apr{p}\iff \left(\frac{\int_I|\Psi'(x)|^pw(x)\,dx}{\int_Iw(x)\,dx}\right)^{\frac{1}{p}}\approx \frac{|\Psi(I)|}{|I|},\quad \text{ for all intervals } I\subset \R,
\end{equation*}
where $\Psi(I)$ is the image of $I$ under $\Psi$.
Indeed,  using \eqref{eq:wwtil} and a change of variable we see that for an interval $J\subset \R$ we have
\begin{align*}
&\left(\frac{1}{|J|}\int_J\wtilp{p}(y)\,dy\right)\left(\frac{1}{|J|}\int_J\wtilp{p}(y)^{1-p'}\,dy\right)^{p-1}\\
 &\quad\quad \quad=
\left(\frac{1}{|\Psi(I)|}\int_{I}|\Psi'(x)|^pw(x)\,dx\right)\left(\frac{1}{|\Psi(I)|}\int_{I}w(x)^{1-p'}\,dx\right)^{p-1},
\end{align*}
where $I=\psii(J).$ Since $w\in \apr{p}$, it holds that
\begin{equation*}
\left(\int_{I} w(x)^{1-p^{\prime}}d x\right)^{p-1}\approx \frac{|I|^p}{\int_I w(x)\,dx}.
\end{equation*}
The above leads to 
\begin{align*}
\left(\frac{1}{|J|}\int_J\wtilp{p}(y)\,dy\right)\left(\frac{1}{|J|}\int_J\wtilp{p}(y)^{1-p'}\,dy\right)^{p-1}\approx
\left(\frac{|I|}{|\Psi(I)|}\right)^p \frac{\int_{I}|\Psi'(x)|^pw(x)\,dx}{\int_I w(x)\,dx},
\end{align*}
from which the desired result follows.

We remark that the above gives in particular that if $w\equiv 1$, then  $\wtilp{p}\in \apr{p}$ iff $|\Psi'|$ satisfies a reverse H\"older inequality with exponent $p;$ this is,
\begin{equation*}
\left(\frac{1}{|I|}\int_I |\Psi'(x)|^p\,dx\right)^{\frac{1}{p}}\approx \frac{1}{|I|}\int_I |\Psi'(x)|\,dx,
\end{equation*}
for all intervals $I\subset \R.$
\end{remark}
\subsection{What does solvability of $\tp{\Psi}{\mu}$ mean?}\label{sec:solvabilitydef} 
Given $0<\alpha<\pi/2$,
$\ntmpm_\alpha$ will denote the 
non-tangential maximal operators given by
 \begin{equation*}
\ntmpm_\alpha (F)(x)= \sup_{z\in \Gamma^{\pm}_\alpha(x)} |F(z)|,\quad x\in \R,
\end{equation*}
where $F$ is a complex-valued function defined in the complex plane  and 
\begin{align*}
&\Gamma^{+}_\alpha(x)=\{z\in \mathbb{C}: \text{Im}(z)>0  \text{ and } |\text{Re}(z)-x|<\tan(\alpha) \text{Im}(z)\},\\
 &\Gamma^{-}_\alpha(x)=\{z\in \mathbb{C}: \overline{z}\in \Gamma^{+}_\alpha(x)\}.
\end{align*} 

For $F^+=F\chi_{\hpp}$ and $x\in \R,$ we say that $F^+$ converges non-tangentially to a complex number $F(x)$ at $x$ if $\lim_{z\to x,z\in \Gamma^{+}_\alpha(x)}F^+(z)=F(x).$ A similar definition applies to $F^-=F\chi_{\hpm}$ using $ \Gamma^{-}_\alpha(x)$.

We next present the definition of solvability of $\tp{\Psi}{\mu}$ in $\lpr{p}{w}.$

\begin{definition}\label{def:pmu}Given a homeomorphism $\Psi:\R\to\R$, $\mu\neq 0$ and a weight $w$ in $\R$, we say that {\it{$\tp{\Psi}{\mu}$ is solvable in $\lpr{p}{w}$}} if, for every $f\in \lpr{p}{w}$, there are harmonic functions $\upm$ in $\hppm$ such that 
\begin{enumerate}[(a)]
\item\label{item:pmua} $\upm$ and $\partial_y\upm$ on $\R$ are the traces of $\upm$ and $\partial_y\upm$, respectively, in the sense of non-tangential convergence,  
\item \label{item:pmub} $\up\circ\Psi=\um$ and $\tpsi(\partial_y \up)-\mu\, \partial_y \um=f$ almost everywhere in $\R$, 
\item \label{item:pmuc} if $0<\alpha<\pi/2,$ then 
\begin{align*}
\|\ntmp_\alpha  \nabla\up\|_{\lpr{p}{\wtilp{p}}}\lesssim \|f\|_{\lpr{p}{w}} \quad \text{and}\quad
\|\ntmm_\alpha  \nabla\um\|_{\lpr{p}{w}}\lesssim \|f\|_{\lpr{p}{w}}.
\end{align*}
\end{enumerate}
We will say that {\it{$\tp{\Psi}{\mu}$ is uniquely solvable in $\lpr{p}{w}$}} if $\tp{\Psi}{\mu}$ is solvable in $\lpr{p}{w}$ and solutions are unique modulo constants.
\end{definition}

\subsection{The Neumann problem in the upper half-plane}\label{sec:solhp:1} The results presented in this section will be useful for some of the proofs in this article.

Consider the classical Neumann boundary value problem in $\hpp:$
\begin{equation}\label{eq:neumann}
  \Delta u =0 \text{ on }\hpp\qquad \text{and}\qquad \partial_y u=f\text{ on }  \R,
  \end{equation}
  where the equality $\partial_y u=f$ is interpreted  in the sense of non-tangential convergence.
  For $f:\R\to \mathbb{C}$ 
define 
\begin{equation}\label{eq:solneumannR2p}
u_f(x,y):=\frac{1}{\pi}\int_{\R}  \log\left(\textstyle{\frac{\sqrt{(x-t)^2+y^2}}{1+|t|}}\right) f(t)\,dt,\quad (x,y)\in \hpp.
 \end{equation}
 We note that the integral on the right-hand side of \eqref{eq:solneumannR2p} is absolutely convergent for all $f$ satisfying $\int_\R |f(x)|(1+|x|)^{-1}\,dx<\infty$; in particular, $u_f$ is well defined and absolutely convergent for all $f\in \lpr{p}{w}$ with $w\in \apr{p}.$
 
 The Neumann problem \eqref{eq:neumann} is solvable in $\lpr{p}{w}$ for $w\in \apr{p};$ more precisely, the following result holds.
 
 \begin{theorem}[Solvability of the Neumann problem in $\lpr{p}{w};$ Theorem 1.3 in \cite{MR4542711}]\label{thm:NeumannSolR2m}
 Let $1<p< \infty$ and $0<\alpha<\pi/2.$  If $w\in \apr{p}$ and $f\in\lpr{p}{w}$,  then $u_f$ is harmonic in $\R_+^2$, $\partial_y u_f=f$ on $\R$ in the sense of non-tangential convergence and 
 \begin{equation*}
\|\ntmp_\alpha(\nabla u_f)\|_{\lpr{p}{w}} \lesssim \|f\|_{\lpr{p}{w}},
\end{equation*}  
where the implicit constant is independent of $f.$ 
 \end{theorem}
 
The following result shows that the solution of the Neumann problem with datum in $\lpr{p}{w}$ for $w\in \apr{p}$ is unique up to a constant.

\begin{theorem}[Uniqueness of solutions for the Neumann problem]\label{thm:uniqueN} Let $f\in \lpr{p}{w}$ with $w\in \apr{p}$. If $u$ is a solution of the Neumann problem \eqref{eq:neumann} with datum $f$  and 
$
\ntmp_\alpha \nabla u\in {\lpr{p}{w}},
$
then there exists a constant $C$ such that
$u=u_f+ C. $
 
\end{theorem}

\begin{proof} Fix $(x, y)\in \hpp$ and $\{z_k\}_k\subset (0,\infty)$  such that $\lim_{k\to \infty}z_k=0.$ Since 
\begin{equation*}
\left|\log\left(\textstyle{\frac{\sqrt{(x-t)^2+y^2}}{(1+|t|)}}\right)\right|\lesssim \frac 1{ 1+|t|},\quad \forall t\in \R,
 \end{equation*}
 we have
\begin{equation*}
\left|  {\partial_2} u( t, z_k) \log\left(\textstyle{\frac{\sqrt{(x-t)^2+y^2}}{(1+|t|)}} \right)\right| \lesssim \frac{\ntmp_\alpha \nabla u(t)}{1+|t|},\quad \forall t\in \R.
 \end{equation*}
The right-hand side is  an integrable function since $\ntmp_\alpha \nabla u\in \lpr{p}{w}$ and $w\in \apr{p}$. Hence, by the dominated convergence theorem and the fact that $\lim_{k\to\infty}\partial_2 u(t, z_k)= f(t)$ for almost every $t\in \R$, we obtain 
\begin{equation*}
\lim_{k\to \infty} \int_{\mathbb R} \partial_2 u(t, z_k)  \log\left(\textstyle{\frac{\sqrt{(x-t)^2+y^2}}{(1+|t|)}}\right) dt =\int_{\mathbb R} f(t) \log\left(\textstyle{\frac{\sqrt{(x-t)^2+y^2}}{(1+|t|)}}\right) dt.
 \end{equation*}

For each $k$, define  
\begin{align*}
U_k(x, y)= \frac{1}{\pi}\int_{\mathbb R}  \log\left(\textstyle{\frac{\sqrt{(x-t)^2+y^2}}{(1+|t|)}} \right)  \partial_2 u(t, z_k) \,dt \quad\text{and} \quad V_k(x, y)= u(x, y+z_k). 
\end{align*}
Setting $W_k(x, y)= U_k(x, y)- V_k(x, y),$ we have  
 \begin{equation*}
\partial_2 W_k(x,0)=\partial_2 U_k(x, 0)- \partial_2 V_k(x, 0)=\partial_2 u(x, z_k)-\partial_2 u(x, z_k)=0 \quad \text{a.e } x\in \R.
\end{equation*}
By the reflection principle, $W_k$ admits a harmonic extension $\widetilde{W}_k$  to $\R^2$ which satisfies  
\begin{equation*}
\ntmp_\alpha \nabla \widetilde{W}^+_k\in \lpr{p}{w}\quad\text{and}\quad \ntmm_\alpha \nabla \widetilde{W}^-_k\in \lpr{p}{w},
\end{equation*}
where $\widetilde{W}^{\pm}_k=\widetilde{W}_k\chi_{\hppm}.$
Let $z=(x, y)\in\hpp;$   by the mean value property applied to $\partial_1 \widetilde W_k$, we have
\begin{align*}
\left| \partial_1 W_k(z)\right| &\le \frac 1{|B(z, r)|}\int_{B(z, r)} \left|\partial_1 \widetilde W_k(s)\right| ds\\
&=
\frac 1{|B(z, r)|}\int_{B(z, r)\cap\hpp} \left|\partial_1 \widetilde{W}_k(s)\right| ds +\frac 1{|B(z, r)|}\int_{B(z, r)\cap\hpm} \left|\partial_1 \widetilde{W}_k(s)\right| ds
\\
&\lesssim
\frac 1{r} \int_{x-r}^{x+r} \ntmp_\alpha\nabla \widetilde W_k^+(t)dt+\frac 1{r} \int_{x-r}^{x+r} \ntmm_\alpha\nabla \widetilde W_k^-(t)dt.
\end{align*}
Since  $w\in \apr{p},$ it follows that
\begin{align*}
\frac 1{r} \int_{x-r}^{x+r} \ntmp_\alpha\nabla \widetilde W_k^+(t)dt&\lesssim
\left( \frac{1}{r}  \int_{x-r}^{x+r}\left|\ntmp_\alpha\nabla  \widetilde W_k^+(t)\right|^p w(t) dt\right)^{1/p} \left( \frac{1}{r} \int_{x-r}^{x+r} w^{1-p'}(t)\,dt\right)^{1/p'}
\\
&\lesssim \frac{1}{r}  \|\ntmp_\alpha\nabla \widetilde W^+_k\|_{\lpr{p}{w}}  \left(  \int_{x-r}^{x+r}  w^{1-p'}(t)\,dt\right)^{1/p'}
\lesssim \frac{\|\ntmp_\alpha\nabla \widetilde W_k^+\|_{\lpr{p}{w}}  }{(\int_{x-r}^{x+r}w(x)\,dx)^{1/p}}, 
\end{align*}
and similarly for the term with $\ntmm_\alpha\nabla \widetilde W_k^-.$
Letting $r\to\infty$, we obtain that $\partial_1W_k(z)=0;$ the same reasoning gives that   $\partial_2 W_k(z)=0$. Hence $W_k$ is constant; that is, there exists $C_k\in\R$ so that
$$
U_k(x, y)=V_k(x, y) + C_k,\quad \forall (x,y)\in \hpp.
$$
Letting $k\to\infty$ we obtain the desired result. 
\end{proof}

 The next two lemmas deal with the boundary values of $u_f$. 
\begin{lemma}[Lemma 3.3 in \cite{CLN2023}]\label{lem:neumanndata} Let $1<p<\infty$, $w\in \apr{p}$ and $f\in \lpr{p}{w}$. 
\begin{enumerate}[(a)]
\item\label{lem:neumanndata:a}   The function  $x\to\int_{\R}  \left|\log\left(\textstyle{\frac{|x-t|}{1+|t|}}\right) f(t)\right|\,dt$ is  locally integrable in $\R.$ Moreover, the function given by
\begin{equation}
\mathcal{B}{f}(x)=\frac{1}{\pi}\int_{\R}  \log\left(\textstyle{\frac{|x-t|}{1+|t|}}\right) f(t)\,dt
\end{equation}
satisfies $(\mathcal{B}{f})'= H f$ in the sense of distributions. 

\item \label{lem:neumanndata:b}  $u_f=\mathcal{B}{f}$ almost everywhere on  $\R$ in the sense of non-tangential convergence. 
\end{enumerate}
\end{lemma}
\begin{remark}\label{re:neumannhpm} Let $1<p<\infty$, $w\in \apr{p}$ and $f\in \lpr{p}{w}$. We note that $u(x,y)=-u_f(x,-y)$ with $x\in \R$ and $y<0$ is a solution of the following Neumann problem in $\hpm:$
\begin{equation*}
  \Delta u =0 \text{ on }\hpm\qquad \text{and}\qquad \partial_y u=f\text{ on }  \R,
  \end{equation*}
Also, Lemma~\ref{lem:neumanndata} gives  
$u=-\mathcal{B}{f}$ almost everywhere on  $\R$ in the sense of non-tangential convergence  and $\partial_xu= -H f$  in the sense of distributions on $\R$.
\end{remark}

\begin{lemma}\label{lem:localint} If  $\Psi:\R\to\R$ is a locally absolutely continuous homeomorphism, $w\in \apr{p}$ and $f\in \lpr{p}{w}$, then  $\mathcal{B}{f}\circ \Psi$ is locally integrable in $\R$.
\end{lemma}
\begin{proof}
Given $M>0$, we will  prove that 
\begin{equation*}
\int_{|x|\le M}\int_{\R} \abs{\log\left(\frac{|\Psi(x)-t|}{1+|t|}\right)}|f(t)|\,dt\,dx<\infty.
\end{equation*}
Let $\bar{M}$ be such that $\abs{\Psi(x)}\le \bar{M}$ for $|x|\le M.$ 

We have
\begin{align*}
\abs{1-\frac{|\Psi(x)-t|}{1+|t|}}&=\abs{\frac{1+|t|-|\Psi(x)-t|}{1+|t|}}\\&\le \frac{1+\abs{|t|-|t-\Psi(x)|}}{1+|t|}\le \frac{1+|\Psi(x)|}{1+|t|}\le \frac{1+\bar{M}}{1+|t|}\quad \text{ for } |x|\le M.
\end{align*}
Choose $\delta>0$  such that 
\begin{equation*}
y>0, \,\abs{1-y}\le \delta\implies \abs{\log y}\lesssim \abs{1-y};
\end{equation*}
then let $K$ be so that 
\begin{equation*}
\abs{t}\ge K\implies \frac{1+\bar{M}}{1+|t|}\le \delta.
\end{equation*}
All of the above implies that 
\begin{equation*}
\abs{t}\ge K\implies \abs{\log\left(\frac{|\Psi(x)-t|}{1+|t|}\right)}\lesssim \frac{1+\bar{M}}{1+|t|}\quad \text{ for } |x|\le M;
\end{equation*}
therefore
\begin{equation*}
\int_{|x|\le M}\int_{|t|>K} \abs{\log\left(\frac{|\Psi(x)-t|}{1+|t|}\right)}|f(t)|\,dt\,dx\lesssim \int_{|x|\le M}\int_{|t|>K}\frac{|f(t)|}{1+|t|}\,dt\,dx<\infty 
\end{equation*}
since $f\in \lpr{p}{w}$ and $w\in \apr{p}.$

We next prove that 
\begin{equation*}
\int_{|x|\le M}\int_{|t|\le K} \abs{\log\left(\frac{|\Psi(x)-t|}{1+|t|}\right)}|f(t)|\,dt\,dx<\infty.
\end{equation*}
We have
\begin{align*}
\int_{|x|\le M}\int_{|t|\le K} &\abs{\log\left(\frac{|\Psi(x)-t|}{1+|t|}\right)}|f(t)|\,dt\,dx \\
&\lesssim 1+\int_{|x|\le M}\int_{|t|\le K} \abs{\log|\Psi(x)-t|}|f(t)|\,dt\,dx\\
&= 1+\int_{|y|\le \bar{M}}\abs{(\Psi^{-1})'(y)}\int_{|t|\le K} \abs{\log|y-t|}|f(t)|\,dt\,dy
\end{align*}
It is then enough to show that
\begin{align}
\int_{|t|\le K} \abs{\log|y-t|}|f(t)|\,dt&\lesssim 1\quad \text{for } |y|\le \bar{M},\label{eq:localint1}\\
\int_{|y|\le \bar{M}}\abs{(\Psi^{-1})'(y)}\,dy&<\infty.\label{eq:localint2}
\end{align}
Regarding \eqref{eq:localint1}, we have
\begin{align*}
\int_{|t|\le K} \abs{\log|y-t|}|f(t)|\,dt\lesssim \|f\|_{\lpr{p}{w}}\left(\int_{|t|\le K}\,\abs{\log |y-t|}^{p'} w(t)^{1-p'}\,dt\right)^{\frac{1}{p'}}.
\end{align*}
Since $w^{1-p'}\in \apr{p'}$, by the reverse H\"older inequality for Muckenhoupt weights,  there exists $r>1$ such that $\int_{|t|\le K} w(t)^{(1-p')r}\,dt<\infty.$ Then, for $|y|\le \bar{M}$, we obtain
\begin{align*}
\int_{|t|\le K} \,\abs{\log|y-t|}^{p'} w(t)^{1-p'}\,dt&\le \left(\int_{|t|\le K} w(t)^{(1-p')r}\,dt\right)^{\frac{1}{r}} \left(\int_{|t|\le K} 
\abs{\log|y-t|}^{p'r'}\,dt\right)^{\frac{1}{r'}}\\
&\le\left(\int_{|t|<K} w(t)^{(1-p')r}\,dt\right)^{\frac{1}{r}} \left(\int_{|s|\le \bar{M}+K} 
\abs{\log|s|}^{p'r'}\,ds\right)^{\frac{1}{r'}}\\
&<\infty.
\end{align*}
As for \eqref{eq:localint2}, let $\bar{\bar{M}}$ be such that $|\Psi^{-1}(y)|\le \bar{\bar{M}}$ for $|y|\le \bar{M};$ we then get
\begin{align*}
\int_{|y|\le \bar{M}}|(\Psi^{-1})'(y)|\,dy=\int_{|x|\le \bar{\bar{M}}}|(\Psi^{-1})'(\Psi(x))|\Psi'(x)|\,dx=\int_{|x|\le \bar{\bar{M}}} 1\,dx<\infty.
\end{align*}
\end{proof}

\subsection{Conformal maps, domains and definition of solvability for  \eqref{eq:TP}}\label{sec:phi} We formalize in this section the definitions of the conformal maps $\phipm$ and the domains $\opm$ given in Section~\ref{sec:intro} as well as the definition of solvability of \eqref{eq:TP}.

Let $\Lambda$  be a rectifiable Jordan curve in the complex plane given parametrically by $ x+ i \gamma (x)$ for $x\in \R$, where $\gamma$ is a real-valued,  and  consider   the domains
  \begin{equation}\label{eq:domains}
\op=\{ z\in \mathbb{C}: \text{Im}(z)> \gamma(\text{Re}(z)) \}, \quad\quad\quad \om=\text{int}({(\op)^c});
\end{equation}
note that $\Lambda=\partial\op=\partial\om.$ The set $\op$ will be called an {\it{upper  graph domain}} and $\om$ will be referred to as a {\it{lower  graph domain}}.

Since $\opm$ are simply connected, they are conformally equivalent to $\hppm.$
 Let $\phipm:\hppm\to \opm$  be conformal maps such that $\phipm(\infty)=\infty.$ Then $\phipm$ extend as  homeomorphisms from $\overline{\hpp}$ onto $\overline{\opm}$ and $\phipm(x)$, $x\in \R$, are absolutely continuous when restricted to any finite interval; in particular, $\phipm'(x)$ exist for almost every $x\in \R$ and are locally integrable. Moreover, $\phipm'(x)\ne 0$ for almost every $x\in\R$,    $\lim_{z\to x}\phipm'(z)=\phipm'(x)$ for almost every $x\in \R$ in the sense of non-tangential convergence. If $\phipm'(x)$ exist and are not zero, then they are vectors tangent to $\La$ at $\phipm(x)$ and $\Re(\phipm')> 0$ almost everywhere on $\R.$  We refer the reader to \cite[proof of Theorem 1.1]{MR556889} for the proof of those properties.

The arc length measure in $\Lambda$ will be denoted by $ds.$ Given a weight $\nu$ in $\Lambda$ (i.e. a non-negative locally integrable (with respect to $ds$) function defined on $\Lambda$), we will denote by $L^p(\Lambda,\nu)$ the space of $p$-integrable functions on $\Lambda$ with respect to $\nu \,ds.$ For future use, we note that 
\begin{equation}\label{eq:normequiv}
\|g\|_{\lpl{p}{\nu}}= \|\tphipm g\|_{\lpr{p}{|\phipm'|^{1-p}\,(\nu\circ \phipm)}}.
\end{equation}

\medskip

When $\gamma$ is a Lipschitz  function we will call $\op$ {\it{an upper Lipschitz graph domain}} and $\om$ a {\it{lower Lipschitz graph domain}}. In this case, it holds that  $|\phipm'|^{-1}\in \apr{2}$ (\cite[Theorem 1.10]{MR556889}). Also, setting $\phipm=\phipm^1+i\,\phipm^2$, we have 
$\Re\left(\frac{1}{\phipm'}\right)\approx \frac{1}{{\phipm^1}'} \approx \frac{1}{|\phipm'|}$ almost everywhere. Indeed, since  
$\phipm^2(x)=\gamma(\phipm^1(x))$ for $x\in \R$, we obtain
\begin{equation}\label{eq:hyperpsi1}
{\phipm^2}'(x)=\gamma'(\phipm^1(x)){\phipm^1}'(x), \qquad \text{a.e. } x\in\R,
\end{equation}
and therefore
\begin{equation}\label{eq:hyperpsi22}
 \quad \Re\left(\frac{1}{\phipm'}\right)=\frac{{\phipm^1}'}{|\phipm'|^2}=  \frac{{\phipm^1}'}{|{\phipm^1}'|^2 (1+ \gamma'(\phipm^1)^2)}.
\end{equation}
The equality \eqref{eq:hyperpsi1} gives $|{\phipm^2}'|\lesssim {\phipm^1}'$ almost everywhere; this and 
\eqref{eq:hyperpsi22} lead to the desired result.
 
 \subsubsection{What does solvability of $\eqref{eq:TP}$ mean?}\label{sec:defsolvabilityTP} 
We start with a definition that extends the idea of non-tangential convergence to domains whose boundary do not posses the cone property (i.e. there does not exist $\beta$ such that the cones of aperture $\beta$ and vertex at points of the boundary are contained in the domain for all points in the boundary).

\begin{definition}\label{def:ntc} Let $\Omega$ be a simply connected domain in the complex plane and $\Phi$ a conformal map from $\hpp$ or $\hpm$ onto $\Omega.$ Given  $\mathcal{R}:\Omega\to \mathbb{C}$,  $r:\partial\Omega \to \mathbb{C}$ and $\xi\in \partial\Omega, $ we say that $\mathcal{R}(z)$ converges to $r(\xi)$ in the sense of {\it{$\Phi$ non-tangential convergence}}  if, for some $0<\alpha<\pi/2$, $\lim \mathcal{R}(z)=r(\xi)$ as $z\to\xi$ with $z\in \Phi(\Gamma_\alpha(\Phi^{-1}(\xi)))$, where $\Gamma_\alpha= \Gamma_\alpha^+$ if the domain of $\Phi$ is $\hpp$ and  $\Gamma_\alpha= \Gamma^-_\alpha$ if the domain of $\Phi$ is $\hpm.$ 

\end{definition}

We note that by \cite[Lemma 1.13]{MR556889}, $\Phi$ non-tangential convergence implies non-tangential convergence when $\Omega$ is an upper or lower Lipschitz graph domain, and the definitions  are equivalent in this setting if the Lipschitz constant of $\partial\Omega$ is less than 1.

We next present the definition of solvability of $\eqref{eq:TP}$ in $\lpl{p}{\nu}.$ 

\begin{definition}\label{def:TPsol} Consider domains $\opm$ as in \eqref{eq:domains} and corresponding conformal maps $\phipm:\hppm\to \opm.$ Given $\mu\neq 0$ and a weight $\nu$ in $\Lambda$, we say that {\it{\eqref{eq:TP} is solvable in $\lpl{p}{\nu}$}} if, for every $g\in \lpl{p}{\nu}$, there are harmonic functions $\vpm$ in $\opm$ such that 
\begin{enumerate}[(a)]
\item \label{item:TPsola} $\vpm$ and $\partial_\normal\vpm$ on $\Lambda$ are the traces of $\vpm$ and $\partial_\normal\vpm$, respectively, in the sense of $\phipm$ non-tangential convergence and $\partial_\normal\vpm\in \lpl{p}{\nu}$,
\item \label{item:TPsolb} $\vp=\vm$ and $\partial_{\normal}\vp- \mu \,\partial_\normal \vm =g$ almost everywhere on $\Lambda$ with respect to arc length, 
\item \label{item:TPsolc} if  $0<\alpha<\pi/2$, then
\begin{align*}
\|\ntmp_\alpha  \nabla(\vp\circ\phip)\|_{\lpr{p}{\wtilp{p}}}\lesssim \|g\|_{\lpl{p}{\nu}} \quad \text{and}\quad
\|\ntmm_\alpha  \nabla(\vm\circ\phim)\|_{\lpr{p}{w}}\lesssim \|g\|_{\lpl{p}{\nu}},
\end{align*}
where $w=|(\phim)'|^{1-p} \,(\nu\circ \phim)$ and $\wtilp{p}$ is as in \eqref{eq:wwtil} with $\Psi=\phipi\circ \phim$.
\end{enumerate}
We will say that {\it{\eqref{eq:TP} is uniquely solvable in $\lpl{p}{\nu}$}} if \eqref{eq:TP} is solvable in $\lpl{p}{\nu}$ and solutions are unique modulo constants.
\end{definition}

\section{General results for the solvability of $\tp{\Psi}{\mu}$ in $\lpr{p}{w}$}\label{sec:tpplanegral}

In this section, we prove Theorem~\ref{thm:TPplanesol}, which states general necessary and sufficient conditions for the solvability and unique solvability of $\tp{\Psi}{\mu}$ in  $\lpr{p}{w}$ in terms of the surjectivity and invertibility, respectively, of the operator $H\tpsi +\mu \,\tpsi H$. This will be useful  for the proofs  of the statements in Section~\ref{sec:solvL2} regarding other conditions that imply solvability of $\tp{\Psi}{\mu}$ in  $\lpr{2}{w}.$ We also present in this section  Corollary~\ref{coro:sym}, which  gives symmetric properties associated to the solvability of $\tp{\Psi}{\mu}$. 

\medskip

The following lemma will be used in the proof of Theorem~\ref{thm:TPplanesol}: 

\begin{lemma}\label{lem:surjetive} Let $\mathcal{X}$ and $\mathcal{Y}$ be Banach spaces and $\mathcal{T}:\mathcal{X}\to \mathcal{Y}$ be a surjective bounded linear operator. Then for every $y\in \mathcal{Y}$, there exists $x_y\in \mathcal{X}$ such that $\mathcal{T}(x_y)=y$ and $\|x_y\|_\mathcal{X}\lesssim \|y\|_\mathcal{Y}.$ \end{lemma}

\begin{proof}

Consider the Banach space  $\bar{\mathcal{X}}=\mathcal{X}/\ker(\mathcal{T})$ with the norm 
$$
\|[x]\|_{\bar{\mathcal{X}}}=\inf\{\|x'\|_\mathcal{X}: \mathcal{T}(x)=\mathcal{T}(x')\},
$$
where $[x]$ denotes the equivalence class of $x.$ Define $\bar{\mathcal{T}}:\bar{\mathcal{X}}\to \mathcal{Y}$ such that $\bar{\mathcal{T}}([x])=\mathcal{T}(x).$ 

It easily follows that $\bar{\mathcal{T}}$ is a bijective bounded linear operator and therefore
$$
\|[x]\|_{\bar{\mathcal{X}}}\approx \|\bar{\mathcal{T}}([x])\|_{\mathcal{Y}}\quad \forall [x]\in \bar{\mathcal{X}}.
$$
Given $y\in \mathcal{Y}$,  $y\neq 0,$ let $x\in \mathcal{X}$ be such that $\mathcal{T}(x)=y$ and choose  $x_y\in [x]$ satisfying  $\|x_y\|_{\mathcal{X}}<2 \|[x]\|_{\bar{\mathcal{X}}};$ then 
$$
\|x_y\|_{\mathcal{X}}\lesssim \|\bar{\mathcal{T}}([x])\|_{\mathcal{Y}}=\|\mathcal{T}(x_y)\|_{\mathcal{Y}}= \|y\|_\mathcal{Y}.
$$
If $y=0,$ the result follows by choosing $x_y=0.$

\end{proof}

We next prove Theorem~\ref{thm:TPplanesol}.

\begin{proof}[Proof of Theorem~\ref{thm:TPplanesol}] We first prove \eqref{item:TPplanesol1} and then \eqref{item:TPplanesol2}.

\medskip

\underline{Proof of \eqref{item:TPplanesol1}.}
We first show that if $H\tpsi +\mu \,\tpsi H:\lpr{p}{\wtilp{p}}\to \lpr{p}{w}$ is surjective then $\tp{\Psi}{\mu}$ is solvable in $\lpr{p}{w}.$ Let $f\in \lpr{p}{w};$ then  Lemma~\ref{lem:surjetive} gives that there exists $h\in \lpr{p}{\wtilp{p}}$  such that $(H\tpsi +\mu \,\tpsi H)(h)=Hf$ and 
\begin{equation}\label{eq:TPplanesol0}
\|h\|_{\lpr{p}{\wtilp{p}}}  \lesssim \|Hf\|_{\lpr{p}{w}},
\end{equation}
with the implicit constant independent of $h$ and $f.$

 Consider the following Neumann problem in the upper-half plane:
\begin{align}\label{eq:TPplanesol1}
\begin{cases}
\Delta \up=0 &  \hbox{in}~\hpp,\\ 
\partial_y \up=h & \hbox{on}~\R.
\end{cases}
\end{align}
Let $\up$ be as given by Theorem~\ref{thm:NeumannSolR2m}. By Lemma~\ref{lem:neumanndata},  $\up$ is well defined in $\R$  and  $\partial_x \up=Hh$ almost everywhere in $\R$.  

Let $U^-$ be the solution of the following Neumann problem as given by Remark~\ref{re:neumannhpm}:
\begin{align}\label{eq:TPplanesol2}
\begin{cases}
\Delta U^-=0 &  \hbox{in}~\hpm,\\ 
\partial_y U^-=\frac{\tpsi h-f}{\mu} & \hbox{on}~\R.
\end{cases}
\end{align}
Again,  $U^-$ is well-defined in $\R$  and   $\partial_x U^-=-H(\frac{\tpsi h-f}{\mu})$ almost everywhere in $\R.$   Since $(H\tpsi +\mu \,\tpsi H)(h)=Hf$, we get
$$\partial_x U^-=-H(\textstyle{\frac{\tpsi h-f}{\mu}})= \tpsi H h\quad \text{ a.e in } \R. $$
Moreover, recalling that  $\Psi'>0$, we obtain
\begin{align*}
&\partial_x (\up\circ \Psi) =(\partial_x \up\circ \Psi )\,\Psi'=\tpsi \partial_x \up = \tpsi H h \quad \text{ a.e in } \R.
\end{align*}
Noting that $\tpsi H h$ is locally integrable, we then have
\begin{equation*}
\partial_x (\up\circ \Psi)=\partial_x U^-  
\end{equation*}
in the sense of distributions. Both, $\up\circ\Psi$ and $U^-$ are locally integrable functions on $\R$ (in particular, they are distributions) by Lemmas \ref{lem:neumanndata} and \ref{lem:localint}, respectively. 
We then conclude that  $\up\circ \Psi=U^-+C$ almost everywhere on $\R$ for some constant $C.$  Defining $\um=U^-+C$, we have  that $\upm$ satisfy $\tp{\Psi}{\mu}$.

We next prove the estimates for the non-tangential maximal operator. Since $\up$ is a solution of the Neumann problem with datum $h$ in $\lpr{p}{\wtilp{p}}$, we have
\begin{equation*}
\|\ntmp_\alpha \nabla\up\|_{\lpr{p}{\wtilp{p}}}\lesssim \|h\|_{\lpr{p}{\wtilp{p}}};
\end{equation*}
moreover, \eqref{eq:TPplanesol0} and the fact that  $w\in \apr{p}$ lead to
\begin{equation}\label{eq:TPplane2}
\|h\|_{\lpr{p}{\wtilp{p}}}  \lesssim \|Hf\|_{\lpr{p}{w}} \lesssim \|f\|_{\lpr{p}{w}}.
\end{equation}
As a consequence, we obtain
\begin{equation*}
\|\ntmp_\alpha \nabla\up\|_{\lpr{p}{\wtilp{p}}}\lesssim \|f\|_{\lpr{p}{w}}.
\end{equation*}
Using that $\um$ is a solution of the Neumann problem with datum $\frac{\tpsi h-f}{\mu}$ in $\lpr{p}{w}$, the fact that $\|\tpsi h\|_{\lpr{p}{w}}=\|h\|_{\lpr{p}{\wtilp{p}}}$, and \eqref{eq:TPplane2}, we see that  
\begin{equation}\label{eq:TPplane3}
\|\ntmm_\alpha  \nabla\um\|_{\lpr{p}{w}}\lesssim \|\textstyle{\frac{\tpsi h-f}{\mu}}\|_{\lpr{p}{w}}\lesssim \|f\|_{\lpr{p}{w}}.
\end{equation}

\bigskip

We next show that if $\tp{\Psi}{\mu}$ is solvable in  $\lpr{p}{w}$, then the operator $H\tpsi +\mu \,\tpsi H:\lpr{p}{\wtilp{p}}\to \lpr{p}{w}$ is surjective. Let $g, f\in\lpr{p}{w}$ be such that $g=Hf$ and consider the solutions $\upm$ of $\tp{\Psi}{\mu}$ with datum $f$; set $h=\partial_y \up$ and note that  $h\in\lpr{p}{\wtilp{p}}$ in view of the estimates satisfied by $\ntmp_\alpha \nabla\up.$ We have
$$\partial_x (\up\circ \Psi)=\partial_x \um\quad \text{ a.e. in } \R;$$
also, since $\up$ is a solution of \eqref{eq:TPplanesol1} and $\um$ is a solution of \eqref{eq:TPplanesol2}, we have $$\partial_x \up=Hh \quad\text{ and }
\quad\partial_x \um=-H\left(\frac{\tpsi h-f}{\mu}\right).$$
This leads to $H(\frac{\tpsi h-f}{\mu})=- \tpsi H h$, from which we get $(H\tpsi +\mu \,\tpsi H)(h)=Hf=g.$ 

\medskip

\underline{Proof of \eqref{item:TPplanesol2}.}
Assume first that $\tp{\Psi}{\mu}$ is uniquely solvable in $\lpr{p}{w}.$ By Part\eqref{item:TPplanesol1},
$H\tpsi +\mu \,\tpsi H:\lpr{p}{\wtilp{p}}\to \lpr{p}{w}$ is surjective.

If $H\tpsi +\mu \,\tpsi H$ is not injective, then there exists $\eta\in \lpr{p}{\wtilp{p}}$ such that $\eta\neq 0$ and $(H\tpsi +\mu\tpsi H)(\eta)=0.$ 
Let $f\in \lpr{p}{w};$ then $Hf\in \lpr{p}{w}$ and by Lemma~\ref{lem:surjetive}, there exists $h_1\in \lpr{p}{\wtilp{p}}$ such that 
$(H\tpsi +\tpsi H)(h_1)=Hf$ and 
\begin{equation*}
\|h_1\|_{\lpr{p}{\wtilp{p}}}  \lesssim \|Hf\|_{\lpr{p}{w}} \lesssim \|f\|_{\lpr{p}{w}}.
\end{equation*}
 Define 
$$
h_2=h_1+\frac{\|f\|_{\lpr{p}{w}}}{\|\eta\|_{\lpr{p}{\wtilp{p}}}}\eta\in \lpr{p}{\wtilp{p}},
$$
which also satisfies $(H\tpsi +\tpsi H)(h_2)=Hf$ and
\begin{equation*}
\|h_2\|_{\lpr{p}{\wtilp{p}}}  \lesssim \|f\|_{\lpr{p}{w}}.
\end{equation*}
The argument in the proof of Part~\eqref{item:TPplanesol1} gives solutions $u_1^{\pm}$ and $u_2^{\pm}$ of $\tp{\Psi}{\mu}$ with datum $f$ which are associated to $h_1$ and $h_2,$ respectively, through \eqref{eq:TPplanesol1} and \eqref{eq:TPplanesol2}. Since $h_1\neq h_2,$ we have that $u_1^{\pm}\neq u_2^{\pm},$ which contradicts the fact that $\tp{\Psi}{\mu}$ is uniquely solvable. We then conclude that $H\tpsi +\mu \,\tpsi H$ is  injective.

\medskip

Conversely, assume that $H\tpsi +\mu \,\tpsi H:\lpr{p}{\wtilp{p}}\to \lpr{p}{w}$ is invertible. Then $\tp{\Psi}{\mu}$ is solvable by Part~\eqref{item:TPplanesol1}. We next show uniqueness of solutions modulo constants.

 If $u_1^\pm$ and $u_2^\pm$ are solutions of $\tp{\Psi}{\mu}$ with datum $f\in \lpr{p}{w},$ then $\upm=u_1^\pm-u_2^\pm$ are harmonic functions in $\hppm$ that satisfy Items \eqref{item:pmua} and \eqref{item:pmub} (with datum zero) of Definition~\ref{def:pmu}; moreover $\ntmp_\alpha\nabla\up\in \lpr{p}{\wtilp{p}}$ and $\ntmm_\alpha\nabla\um\in \lpr{p}{w}$ by Item~\eqref{item:pmuc} of Definition~\ref{def:pmu} for $u_1^{\pm}$ and $u_2^{\pm}$, which imply that $\partial_y\up\in \lpr{p}{\wtilp{p}}$ and  $\partial_y\um\in \lpr{p}{w}.$ Since $\up$ is a solution of the Neumann problem with datum $\partial_y\up,$ we have 
$$\partial_x \up=H (\partial_y\up).$$
Since $\um$ is a solution of the Neumann problem with datum $\partial_y\um=\frac{T_\Psi (\partial_y\up)}{\mu}$ and $\up\circ \Psi=\um$, we have 
  $$T_\Psi H(\partial_y\up)=(\partial_x \up\circ \Psi) \, \Psi'=\partial_x (\up\circ \Psi)=\partial_x \um=-\frac{H(T_\Psi (\partial_y\up))}{\mu}.$$
  We then obtain 
  $$(H\tpsi +\mu \,\tpsi H)(\partial_y\up)=0$$
  and, since $H\tpsi +\mu \,\tpsi H$ is injective, it follows that $\partial_y\up=0.$ Also, $\partial_y\um=\frac{T_\Psi (\partial_y\up)}{\mu}=0.$ By Theorem~\ref{thm:uniqueN}, we have that $\upm$  are constants in $\hppm$. We then conclude that $\tp{\Psi}{\mu}$ is uniquely solvable. 
\end{proof}

\begin{remark}\label{re:Sop} 
Given a homeomorphism $\Psi:\R\to\R$, set 
\begin{equation*}
S= H\tpsii H\tpsi:\lpr{p}{\wtilp{p}}\to\lpr{p}{\wtilp{p}};
\end{equation*}
 then $S$ is a bounded invertible operator. 
Note that
\begin{equation*}
H\tpsi +\mu \,\tpsi H=H\tpsi(I-\mu\, \tpsii H\tpsi H)=H\tpsi(I-\mu\, S^{-1})=H\tpsi(S-\mu\, I)S^{-1}.
\end{equation*}
We then have that $H\tpsi +\mu \,\tpsi H:\lpr{p}{\wtilp{p}}\to \lpr{p}{w}$ is surjective (injective) if and only if $S-\mu\, I :\lpr{p}{\wtilp{p}}\to \lpr{p}{\wtilp{p}}$ is surjective (injective). 

\medskip

We next note that, since $S-\mu\, I =\mu(\frac{1}{\mu} S-I),$ then $S-\mu\, I$ is invertible if $\|S\|<\abs{\mu}.$ Also, since $S-\mu\, I =S(I-\mu S^{-1}),$ then $S-\mu\, I$ is invertible if $\abs{\mu}<\|S^{-1}\|^{-1}.$ As a consequence $\tp{\Psi}{\mu}$ is solvable for $\abs{\mu}$ sufficiently small and for $\abs{\mu}$ sufficiently large.

\end{remark}

\begin{remark}\label{re:S-I}
By \cite[Section 2]{MR2091359}, the transmission problem  \eqref{eq:TP} is uniquely solvable  in $\lplnw{p}$  when $\mu=1,$  $1<p<\infty$ and $\opm$ are upper and lower Lipschitz graph domains. We can then apply  Theorem~\ref{thm:tplwgral} to conclude that  $\tp{\Psi}{1}$ with $\Psi=\phipi\circ \phim$ (assuming $\Psi$ locally absolutely continuous) is uniquely solvable in $\lpr{p}{|\phim'|^{1-p}}.$ Noting that for $w=|\phim'|^{1-p},$ we have $\wtilp{p}=|\phip'|^{1-p}$, assuming both weights are in $\apr{p}$,   Theorem~\ref{thm:TPplanesol} and Remark~\ref{re:Sop} give that the operator $S-I:\lpr{p}{|\phip'|^{1-p}}\to \lpr{p}{|\phip'|^{1-p}}$ is invertible. This is the case for $p=2$ since $|\phim'|^{-1}$ and $|\phip'|^{-1}$ are in $\apr{2}.$ 
\end{remark}

Theorem~\ref{thm:TPplanesol}  leads to the following symmetry properties for  the solvability of $\tp{\Psi}{\mu}$ in $\lpr{2}{w}.$
\begin{corollary}\label{coro:sym} Let $\mu\neq 0,$ $1<p<\infty,$ and $\Psi:\R\to\R$ be a locally absolutely continuous homeomorphism with $\Psi'>0$ almost everywhere; assume  $w\in \apr{p}$  is such that $\wtilp{p}\in \apr{p}$. It then holds that 
\begin{enumerate}[(a)]
\item\label{coro:sym:a} $\tp{\Psi}{\mu}$  is uniquely solvable in  $\lpr{p}{w}$ if and only if $\tp{\Psi}{1/\mu}$ is uniquely solvable in $\lpr{p}{w}$.
\item \label{coro:sym:b}
$\tp{\Psi}{\mu}$  is uniquely solvable in  $\lpr{p}{w}$ if and only if $\tp{\Psi^{-1}}{\mu}$ is uniquely solvable in $\lpr{p}{\wtilp{p}}$.
\end{enumerate}
\end{corollary}

\begin{proof}
The result will follow by doing simple manipulations of the operator $H\tpsi +\mu \,\tpsi H$ and using the characterization of solvability of $P_\Psi(\mu)$ given in Theorem~\ref{thm:TPplanesol}. 

\medskip

\underline{Proof of \eqref{coro:sym:a}.} We have:
$$
H\tpsi +\mu \,\tpsi H=- \mu H \big( H \tpsi + \tfrac{1}{\mu} \tpsi H\big) H.
$$
Since $H:\lpr{p}{\wtilp{p}} \to \lpr{p}{\wtilp{p}}$ is invertible,  it follows that $H\tpsi +\mu \,\tpsi H:\lpr{p}{\wtilp{p}}\to \lpr{p}{w}$ is invertible if and only if $H \tpsi + \tfrac{1}{\mu} \tpsi H: \lpr{p}{\wtilp{p}} \to \lpr{p}{w}$ is invertible.

\medskip

\underline{Proof of \eqref{coro:sym:b}.} Similarly:
$$
H\tpsi +\mu \,\tpsi H= \mu \tpsi \big( H \tpsii + \tfrac{1}{\mu} \tpsii H\big) \tpsi.
$$
Since $\tpsi: \lpr{p}{\wtilp{p}}\to \lpr{p}{w}$ is invertible, it follows that $H\tpsi +\mu \,\tpsi H:\lpr{p}{\wtilp{p}}\to \lpr{p}{w}$ is invertible if and only if $H \tpsii + \tfrac{1}{\mu} \tpsii H: \lpr{p}{w} \to \lpr{p}{\wtilp{p}}$ is invertible. Finally, by Part~\eqref{coro:sym:a}, we get the result.
\end{proof}

We end this section with a lemma that gives sufficient conditions for a family of linear operators to be invertible; this result will be used in Section~\ref{sec:solvL2}. See \cite{MR1224587}, we include the proof for the sake of completeness. 

\begin{lemma}\label{lem:LipTmu} Let $\mathcal{X}$ and $\mathcal{Y}$ be Banach spaces and $I\subset \R$ be an open interval. Consider a family $\{\mathcal{T}_\mu\}_{\mu\in I}\subset \mathcal{B}(\mathcal{X},\mathcal{Y})$  such that $\mathcal{T}_{\mu_0}$ is invertible for some $\mu_0\in I$, $\mathcal{T}_\mu(\mathcal{X})$ is closed in $\mathcal{Y},$  $\dim{(\ker} (\mathcal{T}_\mu))=0$ for all $\mu\in I$, and $\mu\to \mathcal{T}_\mu$ is continuous in I.
  Then $\mathcal{T}_\mu$ is invertible for all $\mu\in I.$
\end{lemma} 
\begin{proof} In view of the fact that $\dim{(\ker} (\mathcal{T}_\mu))=0,$ we only need to show that $\mathcal{T}_\mu$ is surjective.
Define 
\begin{equation*}
\mathcal{U}=\{ T\in \mathcal{B}(\mathcal{X},\mathcal{Y}): T(\mathcal{X})\text{ is closed in }\mathcal{Y}\text{ and }\dim{(\ker} (T))<\infty\}.
\end{equation*}
Note that $\{\mathcal{T}_\mu\}_{\mu\in I}\subset \mathcal{U}.$ Let $\text{Ind}:\mathcal{U}\to \Z\cup\{-\infty\}$ be the index function given by
\begin{equation}
\text{Ind}(T)=\dim{(\ker}(T)) -\dim\left(\mathcal{Y}/T(\mathcal{X})\right),\quad T\in \mathcal{U}.
\end{equation}
   By \cite[Theorem 2.2]{MR1224587}, $\mathcal{U}$ is open and Ind is continuous in $\mathcal{U}$. This and the continuity of $\mu\to \mathcal{T}_\mu$ in $I$ imply that the function $f:I\to \Z\cup \{-\infty\}$ defined by $f(\mu)=\text{Ind}(\mathcal{T}_\mu)$ is continuous in $I$. Since $f(\mu_0)=0$ we obtain  that  $f(\mu)=0$ for $\mu\in I.$

We then conclude that $\dim\left(\mathcal{Y}/\mathcal{T}_\mu(\mathcal{X})\right)=\dim{(\ker}(\mathcal{T}_\mu))-\text{Ind}(\mathcal{T}_\mu)=0;$ thus $\mathcal{T}_\mu(\mathcal{X})$ is dense in $\mathcal{Y}$, which along with the assumption that $\mathcal{T}_\mu(\mathcal{X})$ is closed, implies that $\mathcal{T}_\mu$ is surjective. 

\end{proof}

\section{Solvability results for  $\tp{\Psi}{\mu}$ in $\lpr{2}{w}$}\label{sec:solvL2}

In this section, we present different settings that lead to solvability results for $\tp{\Psi}{\mu}$ in weighted $L^2$ spaces, which are used in Section~\ref{sec:solTP} to study solvability of the transmission problem~\eqref{eq:TP} and are interesting in their own right.

In Section~\ref{sec:solvL2_a}, we prove Theorem~\ref{thm:hyper} and state  Theorems
\ref{thm:hypergral} and \ref{thm:hilb}; all of these results give sufficient conditions on the homeomorphism $\Psi$ for the solvability of $\tp{\Psi}{\mu}$  in $\lpr{2}{\frac{1}{\Psi'}};$ we also present examples associated to domains that include an infinite staircase and symmetric cones, as well as an example related to the Helson-Szeg\"o representation of $A_2$ weights. In Section~\ref{sec:solvL2_a}, we prove Theorem~\ref{thm:welding} , which   deals with the solvability of $\tp{\Psi}{\mu}$  in $\lpr{2}{|\phim'|^{-1}}$ when $\Psi=\phip^{-1}\circ \phim$. A main tool in the proofs of all these results is the use of a Rellich identity for the Hilbert transform, which we present in Section~\ref{sec:rellich}.

\subsection{Rellich identity}\label{sec:rellich}
For a real-valued Schwarz function $f$, the following formula known in the literature as ``the magic formula" holds true (see \cite[(5.1.23), p. 320]{MR3243734}):
$$
(Hf)^2-f^2= 2H(fHf).
$$ 
From here, it can be easily deduced that if $v$ is a weight for which $H v$ is well defined and $f$ is a real-valued Schwarz function, then
\begin{equation}\label{magic}
\int_\mathbb R (Hf)^2 v  dx = \int_\mathbb R f^2 v  dx - 2 \int_\mathbb R f \,Hf  H v dx.
\end{equation}
A second formula of this type was proved in \cite{MR4674966} using Rellich identity and connections with the Neumann problem on Lipschitz graph domains. 
\begin{theorem}[Theorem 1.2 in \cite{MR4674966}] \label{thm:rellich}
Let $\Phi$ be a conformal map from $\hpp$ onto an  upper Lipschitz graph domain and  
 $f\in \lpr{2}{|\Phi'|^{-1}}$ be real-valued. Then
\begin{equation}\label{eq:rellich}
\int_\mathbb R (Hf)^2 \re dx = \int_\mathbb R f^2 \re dx - 2 \int_\mathbb R f \,Hf \im dx.
\end{equation}
\end{theorem}

\begin{remark} If $\Phi$ is a conformal map from $\hpm$ onto a lower Lipschitz graph domain, the Rellich formula \eqref{eq:rellich} becomes,
\begin{equation}\label{eq:rellichm}
\int_\mathbb R (Hf)^2 \re dx = \int_\mathbb R f^2 \re dx + 2 \int_\mathbb R f \,Hf \im dx.
\end{equation}
This follows by applying \eqref{eq:rellich} to the conformal map $-\Phi(-x,-y)$, which satisfies the hypothesis of Theorem~\ref{thm:rellich}.
\end{remark}

\subsection{Solvability of $\tp{\Psi}{\mu}$ in $\lpr{2}{\frac{1}{\Psi'}}$}\label{sec:solvL2_a}

 Our results in this section  use the Rellich identity \eqref{eq:rellich} to obtain solvability  of $\tp{\Psi}{\mu}$ in $\lpr{2}{\frac{1}{\Psi'}}$ under different assumptions on $\Psi.$ We start with the proof of Theorem~\ref{thm:hyper}.

\begin{proof}[Proof of Theorem~\ref{thm:hyper}]
Note that if $w=\frac{1}{\Psi'}$,   then $\wtil=1$; also, $w\in \apr{2}$ since $|\Phi'|^{-1}\in \apr{2}$ and $|\Phi'|^{-1}\sim \re$. 
By Theorem~\ref{thm:TPplanesol}, Remark~\ref{re:Sop} and Part~\eqref{coro:sym:a} of Corollary~\ref{coro:sym}, it is enough to see that the operator $S-\mu I:\lprnw{2}\to \lprnw{2}$ is invertible for $0<|\mu|<1$ satisfying \eqref{eq:muk}.

Set  $h= \Psi' \,{\im}$ and let $f\in \lprnw{2};$ without loss of generality we may assume that $f$ is real-valued. Noting that $\tpsi f\in \lpr{2}{\frac{1}{\Psi'}}= \lpr{2}{|\Phi'|^{-1}}$, we  apply \eqref{eq:rellich}  to the function $\tpsi f$ to obtain 
\begin{eqnarray*}
\int_\mathbb R (H\tpsi f)^2 \frac 1{\Psi'} dx &=& \int_\mathbb R (H\tpsi f)^2 \re dx
\\
&=&  \int_\mathbb R (\tpsi f)^2 \re dx - 2 \int_\mathbb R \tpsi f \,H\tpsi f \im dx
\\
&=&  \int_\mathbb R (\tpsi f)^2  \frac 1{\Psi'} dx - 2 \int_\mathbb R \tpsi f \,H \tpsi f  \frac{h}{\Psi'}dx.
\end{eqnarray*}
A change of variables  gives  $\int_{\R} F G \frac{1}{\Psi'}\,dx=\int_\R \tpsii F\, \tpsii G \,dx$, and therefore it follows that
$$
\int_\mathbb R (\tpsii H\tpsi f)^2   dx =\int_\mathbb R f^2    dx - 2 \int_\mathbb R  f \,\tpsii\,H\tpsi f  \,( h\circ \Psi^{-1})dx;
$$
equivalently, 
$$
\int_\mathbb R (HS f)^2   dx =  \int_\mathbb R f^2    dx + 2 \int_\mathbb R  f \,HS f  \, ( h\circ \Psi^{-1}) dx.
$$
This leads to 
\begin{eqnarray}
& &\int_\mathbb R H(S-\mu I)f\, H(S+\mu I)f   dx = \int_\mathbb R (HS f)^2 dx - \mu^2 \int_\mathbb R (Hf)^2 dx\nonumber
\\
&= &  \int_\mathbb R f^2    dx + 2  \int_\mathbb R  f \,HS f   \, ( h\circ \Psi^{-1})dx- \mu^2 \int_\mathbb R (Hf)^2   dx\nonumber
\\
&=& \left(1 -\mu^2\right)\int_\mathbb R f^2   dx 
+ 2 \int_\mathbb R  f \,H[S-\mu I] f  \, ( h\circ \Psi^{-1}) dx + 2\mu \int_\mathbb R  f Hf \, ( h\circ \Psi^{-1}) dx, \label{eq:hyper}
\end{eqnarray}
that is
\begin{align}
\left(1-\mu^2\right)&\int_\mathbb R f^2   dx+ 2 \mu  \int_\mathbb R  f Hf  \,( h\circ \Psi^{-1}) dx\label{eq:hyper1}\\
&=\int_\mathbb R H(S-\mu I)f\, H(S+\mu I)f    dx - 2 \int_\mathbb R  f \,H(S-\mu I) f \, ( h\circ \Psi^{-1})  dx\label{eq:hyper2}
\end{align}
Observe that  \eqref{eq:hyper2} is controlled up to a constant by
$$
||(S-\mu I)f||_{\lprnw{2}} ||f||_{\lprnw{2}},
$$
and, since $||h\circ \Psi^{-1}||_\infty \le \kpsi$, the second term of \eqref{eq:hyper1} satisfies
$$
\left|  2 \mu\int_\R  f Hf  ( h\circ \Psi^{-1}) dx\right| \le 2 |\mu| \kpsi ||f||_{\lprnw{2}}^2,
$$
giving that
$$
 2 \mu\int_\R  f Hf  \,(h\circ \Psi^{-1}) dx\ge -2 |\mu| \kpsi ||f||_{\lprnw{2}}^2.
$$
As a consequence,
$$
(1- \mu^2 - 2\,\kpsi\,|\mu|) ||f||_{\lprnw{2}} \lesssim ||(S-\mu I)f||_{\lprnw{2}},
$$
and hence, given \eqref{eq:muk}, we conclude that  $S-\mu I$ is injective and has closed range. An application of Lemma~\ref{lem:LipTmu} with $\mu_0=0$ gives that $S-\mu I$ is invertible.

\end{proof}

\subsubsection{Applications of Theorem~\ref{thm:hyper}}\label{sec:applhyper}

We next present examples of homeomorphisms $\Psi$ that satisfy the hypothesis of  Theorem~\ref{thm:hyper}.

\label{ex:hyperpsi}
  Let  $\Phi=\Phi^1+i \,\Phi^2$   be a conformal map from $\hpp$ onto an upper Lipschitz graph domain associated to a Lipschitz curve $t+i \gamma(t)$ for $t\in \R.$  By \eqref{eq:hyperpsi22},  if $\Psi$ is given as in \eqref{eq:hyper0}, we have
$$
\frac 1{\Psi'}= \re= \frac{{\Phi^1}'}{|\Phi'|^2}=  \frac{{\Phi^1}'}{|{\Phi^1}'|^2 (1+ \gamma'(\Phi^1(x))^2)},
$$ 
and since ${\Phi^1}'> 0$ almost everywhere, we obtain that
\begin{equation}\label{eq:hyperpsi2}
\Psi'(x)= {\Phi^1}' (x) (1+ \gamma'(\Phi^1(x))^2).
\end{equation}

\bigskip

We present specific examples of the above:

\medskip

\begin{enumerate}[(a)]
\item
\underline{Example related to the infinite staircase:}  Let $\Omega$ to be the ``infinite rotated staircase" with interior angles alternately equal to  $\pi/2$ and $3\pi/2$ as shown in Figure \ref{fig:staircase}. 

\bigskip

\bigskip

\bigskip

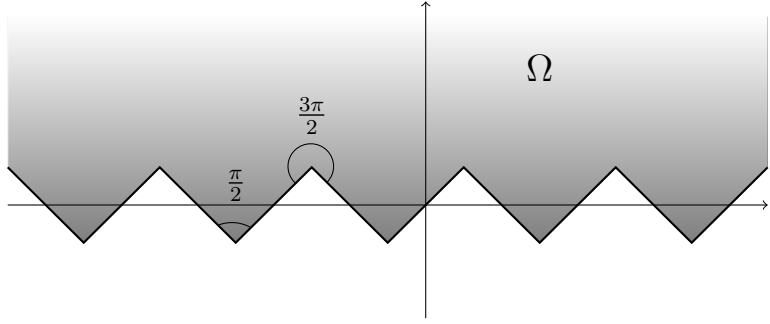
\begin{figure}[h!]
 \begin{tikzpicture}
 \filldraw[draw=white,bottom color=gray, top color=white] (10,0)--(9,-1)--(8,0)--(7,-1)--(6,0)--(5,-1)--(4,0)--(3, -1)--(2,0)--(1,-1)--(0,0)--(0,2)--(10,2);
\draw[->] (0,-0.5)--(10,-0.5);
\draw[->] (5.5,-2)--(5.5,2.2);
\draw[-,thick] (0,0)--(1,-1)--(2,0)--(3, -1)--(4,0)--(5,-1)--(6,0)--(7,-1)--(8,0)--(9,-1)--(10,0);
\node [above] at (7,1) {\large$\Omega$};
\draw (4.2,-0.2) arc (-45:230:0.3);
\node[above] at (4,0.3) {{$\frac{3\pi}{2}$}};
\draw (3.2,-0.8) arc (60:110:0.5);
\node[above] at (3,-0.6) {{$\frac{\pi}{2}$}};
\end{tikzpicture}
\caption{Infinite staircase}
\label{fig:staircase}
\end{figure}
In this case, a conformal map $\Phi:\hpp\to \Omega$ is given by 
$$
\Phi(z)= \frac{\sqrt 2\pi}{4}  e^{i\frac \pi 4} + e^{-i\frac \pi 4}  \int_0^z  \sqrt{\tan y} \, dy,$$
 which satisfies 
$$\Phi'(x)=e^{-i\frac \pi 4}   \sqrt{|\tan x|},\quad x\in \R.$$
Then, if $\Psi:\R\to\R$ is a locally absolutely continuous homeomorphism such that $\frac{1}{\Psi'}=\re$, using \eqref{eq:hyperpsi2} and the fact that $|\gamma'(t)|=1$ almost everywhere give
$$
\Psi'(x)=  \sqrt{2 |\tan x|}.
$$
Also, since $\Re(\Phi')(x)=-\Im(\Phi')(x)$ for $x\in \R$, it follows that $\kpsi=1.$ Then $\tp{\Psi}{\mu}$ is uniquely solvable in $\lpr{2}{\frac{1}{\Psi'}}$ for  $\mu\ne 0$ such that $|\mu|< \sqrt{2}-1$ or $|\mu|>\frac{1}{\sqrt{2}-1}.$

\bigskip

\item \underline{Example  related to a symmetric infinite sector of aperture $\alpha$:} Let $\Omega$ be a cone with aperture $\alpha \pi$, with $\alpha\in (0,2)$, which is symmetric about the imaginary axis (see Figure~\ref{fig2}). Consider the conformal map $\Phi : \hpp \to \Omega$ such that
\begin{equation}\label{map1}
\Phi(z) = e^{i \tfrac{(1-\alpha)}{2} \pi} z^\alpha = i e^{-i \tfrac{\alpha}{2} \pi} e^{\alpha ( \log |z| + i \A(z) )},
\end{equation}
where we chose the branch cut $\{ i y : y \leq 0\}$, so that $\Phi$ is analytic on $\hpp$. 

\bigskip

\medskip

\begin{figure}[h!]
 \begin{center}
\begin{tikzpicture}[scale=.48]
\filldraw[draw=white,bottom color=lightgray, top color=white] (0, 5) -- (3,0)--(6, 5) ; 
\draw [ ->] (-1,0) -- (7,0); 
\draw [ ->] (3,-.2) -- (3,5);  
\draw [ thick] (6, 5) -- (3,0);  
\draw [ thick] (0, 5) -- (3,0); 
\draw (3.5,0.9) arc (55:120:1); 
\node[above] at (3,1.2) {{$\alpha \pi $}}; 
\node[below] at (4,5) {{\large$\Omega$}}; 
\end{tikzpicture}
\caption{Symmetric cone with aperture $\alpha \pi$.}
\label{fig2}
\end{center}
\end{figure}
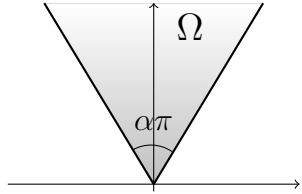
It follows that $\re=\alpha^{-1}\sin(\frac{\alpha\pi}{2})|x|^{1-\alpha}$  and therefore, if $\Psi:\R\to\R$ is a locally absolutely continuous homeomorphism such that $\frac{1}{\Psi'}=\re$, we have $\Psi'(x)=\frac{\alpha}{\sin(\frac{\alpha\pi}{2})}|x|^{\alpha-1}$. Note that $\Psi'\in \apr{2}$ if and only if $0<\alpha<2$, as expected. We have $\im=-\alpha^{-1}\cos(\frac{\alpha\pi}{2})\sgn(x)|x|^{1-\alpha}$ which together with the expression for $\Psi'$ gives that $\kpsi=|\cot(\frac{\alpha\pi}{2})|.$ In particular, if $\alpha=\frac{1}{2}$, then $k_\Psi=1$ and $\tp{\Psi}{\mu}$ is uniquely solvable in  $\lpr{2}{\frac{1}{\Psi'}}$ for  $\mu\ne 0$ such that $|\mu|< \sqrt{2}-1$ or $|\mu|>\frac{1}{\sqrt{2}-1}.$

\bigskip

\item  \underline{Examples of homeomorphisms using the Helson-Szeg\"o representation of $A_2$ weights:} If $f$ is such that $\|f\|_{\infty}<\frac{\pi}{2}$, by the Helson-Szeg\"o representation of $A_2$ weights, there is a conformal map $\Phi$ from $\hpp$ onto an upper  Lipschitz graph domain such that $\Phi'= e^{Hf}e^{-i f}$ on $\R$ (see proof of \cite[Lemma 1.11]{MR556889}). 
We have that
$$
 \Phi_1' (x)= e^{Hf} \cos f, \qquad \Phi_2' (x)= -e^{Hf} \sin f
 $$
 and, recalling \eqref{eq:hyperpsi1}, we obtain
 $$
 \gamma'(\Phi_1(x))= -\tan f(x).
 $$
 Thus, if $\Psi:\R\to\R$ is a homeomorphism such that $\frac{1}{\Psi'}=\re$, then  \eqref{eq:hyperpsi2} gives
 $$ 
\Psi'(x)= e^{Hf} \cos f \,(1+ \tan^2 f)= \frac{e^{Hf} }{\cos f}.
$$
In this case, $\kpsi=||\tan f||_\infty$. 

\end{enumerate}

 \subsubsection{Extensions of Theorem~\ref{thm:hyper}}
 In general, given a homeomorphism $\Psi:\R\to\R$, it is not easy to see whether condition  \eqref{eq:hyper0} holds. However, if  $\Psi'\in \apr{2}$,  using the Helson-Szeg\"o representation of $A_2$ weights, it follows that there exists a conformal map $\Phi$ from $\hpp$ onto an upper Lipschitz graph domain and a constant $\apsi>0$  such that 
 $$
 \frac {\apsi}{\Psi'}\le  \re \le   \frac {1}{\Psi'}.
$$ 
With obvious modification  in the proof of Theorem~\ref{thm:hyper}, we get the following result.

\begin{theorem}\label{thm:hypergral} 
Let $\Psi:\R\to \R$ be a locally absolutely continuous homeomorphism such that $\Psi'>0$ almost everywhere  and  $\Psi'\in \apr{2}.$ Consider a conformal map  $\Phi$ from $\hpp$ onto an upper Lipschitz graph domain such that
\begin{equation}\label{eq:hypergral0}
 \frac {\apsi}{\Psi'}\le  \re \le   \frac {1}{\Psi'}
\end{equation}
for some   positive constant $\apsi$. Define 
$$
\kpsi:=\left\| \Psi'\,{\im}\right\|_{\lprnw{\infty}}.
$$
Then  for every $0<|\mu|<1$ such that 
\begin{equation*}
\apsi-\mu^2 - 2\,\kpsi\,|\mu|  >0,
\end{equation*}
the transmission problems $\tp{\Psi}{\mu}$ and $\tp{\Psi}{1/\mu}$  are uniquely solvable in $\lpr{2}{\frac{1}{\Psi'}}.$ 
\end{theorem}

We note that, if $L$ is the Lipschitz constant associated to the boundary of $\Phi(\hpp)$ with  $\Phi$ as in Theorem~\ref{thm:hypergral}, then 
\begin{align*}
\kpsi=\left\|\Psi'\re\,\frac{{\im}}{\re}\right\|_{\lprnw{\infty}}&\le \left\| \Psi'\re\right\|_{\lprnw{\infty}}\left\|\frac{{\im}}{\re}\right\|_{\lprnw{\infty}}\\&= \left\| \Psi'\re\right\|_{\lprnw{\infty}}\left\|\frac{{\Im(\Phi')}}{\Re(\Phi')}\right\|_{\lprnw{\infty}}\le L.
\end{align*}

An application of Theorem~\ref{thm:hypergral}, associated to  the hyperbola $y=1/x,$  $x>0,$ is presented in Section~\ref{sec:hyperbola}.

\bigskip

Also, with the same proof as in Theorem \ref{thm:hyper} and using formula \eqref{magic},  we obtain the following result.

\begin{theorem} \label{thm:hilb} Let $\Psi:\R\to \R$ be a locally absolutely continuous homeomorphism such that $\Psi'>0$ almost everywhere, $\Psi'\in \apr{2}$ and 
$$
\left|H\left(\frac 1{\Psi'}\right)\right| \le \frac{\cpsi}{\Psi'}
$$ 
 for some positive constant $\cpsi.$ Then  for every $0<|\mu|<1$  satisfying 
 \begin{equation*}
 1-\mu^2 -2\,\cpsi\,|\mu|>0,
\end{equation*}
the transmission problems $\tp{\Psi}{\mu}$ and $\tp{\Psi}{1/\mu}$  are uniquely solvable in  $\lpr{2}{\frac{1}{\Psi'}}$.
\end{theorem}

\subsubsection{An application of Theorem~\ref{thm:hilb}: perturbation of the identity} \label{sec:hilbex}

Given $\varepsilon>0$ consider $$\Theta_\varepsilon(x)=x+\varepsilon \arctan(x);$$ then 
\begin{equation*}
\Theta_\varepsilon'(x)=1+\frac{\varepsilon}{1+x^2},\quad \frac{1}{\Theta_\varepsilon'(x)}=1-\frac{\varepsilon}{1+\varepsilon+x^2}\quad
\text{ and } \quad H\left(\frac{1}{\Theta_\varepsilon'}\right)(x)=-\frac{\varepsilon}{\sqrt{1+\varepsilon}}\,\frac{x}{1+\varepsilon+x^2}.
\end{equation*}
We have
\begin{align*}
\abs{H\left(\frac{1}{\Theta_\varepsilon'}\right)(x)\,\Theta_\varepsilon'(x)}=\abs{\frac{\varepsilon}{\sqrt{1+\varepsilon}}\,\frac{x}{1+x^2}}\le \frac{\varepsilon}{2\sqrt{1+\varepsilon}}.
\end{align*}
Theorem~\ref{thm:hilb} implies that   $\tp{\Theta_\varepsilon}{\mu}$ and $\tp{\Theta_\varepsilon}{1/\mu}$ are uniquely solvable in  $\lprnw{2}=\lpr{2}{\frac{1}{\Theta_\varepsilon'}}$ (note that $\frac{1}{\Theta_\varepsilon'} $ behaves like a constant)
 for $\mu$ such that $1-\mu^2- \frac{\varepsilon}{\sqrt{1+\varepsilon}}\,|\mu|>0,$ this is, $\abs{\mu}\le \frac{1}{\sqrt{1+\varepsilon}}.$

\subsection{Solvability of $\tp{\Psi}{\mu}$ in $\lpr{2}{|\phim'|^{-1}}$ with $\Psi=\phip^{-1}\circ \phim$}\label{sec:solvL2_b}
In this section, we prove Theorem~\ref{thm:welding}. The next lemma will be used in its proof.
\begin{lemma}\label{lem:cvar} If $\Psi$ is as above, we have
\begin{align*}
\int_\R \tpsii f\, \tpsii g \,\Re \left(\frac 1{\phip'} \right)dx &=\int_\R  f\, g \,\Re \left(\frac 1{\phim'} \right)dx, \\
\int_\R \tpsii f \,\tpsii g \,\Im \left(\frac 1{\phip'} \right)dx &=\int_\R f \,g \,\Im \left(\frac 1{\phim'} \right)dx.
\end{align*}

\end{lemma}
\begin{proof} We will only prove the first equality. 
\begin{align*}
\int_\R &\tpsii f\, \tpsii g \, \Re \left(\frac 1{\phip'} \right)dx\\
&=\int_\R |(\Psi^{-1})'(x)|^2 f(\Psi^{-1}(x)) g(\Psi^{-1}(x))   \Re \left(\frac {\phip'(x)}{|\phip'(x)|^2} \right)dx
\\
&=\int_\R |(\phim^{-1})'(\phip(x)) \phip'(x)|^2 \,f(\phim^{-1}(\phip(x)))\,g(\phim^{-1}(\phip(x)))\,   \Re \left(\frac {\phip'(x)}{|\phip'(x)|^2} \right)dx
\\
&=\int_{\Lambda} |(\phim^{-1})'(y)|^2 \, f(\phim^{-1}(y)) \,g(\phim^{-1}(y))\,  \Re \left(\frac {\phip'(\phip^{-1}(y))}{|\phip'(\phip^{-1}(y))|} \right)dy
\\
&=\int_{\Lambda} |(\phim^{-1})'(y)|^2 \, f(\phim^{-1}(y)) \,g(\phim^{-1}(y))\,  \Re \left(\frac {\phim'(\phim^{-1}(y))}{|\phim'(\phim^{-1}(y))|} \right)dy\\
&=
\int_\R f\,g\,\Re \left(\frac 1{\phim'} \right)dx.
\end{align*}
\end{proof}

\begin{proof}[Proof of Theorem~\ref{thm:welding}]
 By Theorem~\ref{thm:TPplanesol}, Remarks~\ref{re:Sop} and \ref{re:S-I} and Part~\eqref{coro:sym:a} of Corollary~\ref{coro:sym}, it is enough to show that $S-\mu I,$ where $S=H\tpsii H\tpsi$, is invertible  in $\lpr{2}{|\phip'|^{-1}}$ for $0<\mu<1$, noting that $\wtil=|\phip'|^{-1}$ for $w=|\phim'|^{-1}$ and both belong to $\apr{2}.$

Defining
\begin{equation*}
A= 2\int_\mathbb R (S-\mu I)f \,(\tpsii  H\tpsi  f) \Im \left(\frac 1{\phip'} \right)dx,
\end{equation*}
we have
\begin{eqnarray*}
A&=&
2\int_\R Sf (\tpsii  H \tpsi  f) \Im \left(\frac 1{\phip'} \right)dx -2\mu\int_\R  f (\tpsii  H \tpsi  f) \Im \left(\frac 1{\phip'} \right) dx
\\
&=&
-2\int_\mathbb R Sf \, HSf \Im \left(\frac 1{\phip'} \right)dx - 2\mu\int_\mathbb R \tpsi  f ( H \tpsi f) \Im \left(\frac 1{\phim'} \right) dx, 
\end{eqnarray*}
where in the last equality we used Lemma~\ref{lem:cvar}. Using the Rellich identities \eqref{eq:rellich} and \eqref{eq:rellichm}, and Lemma~\ref{lem:cvar}, it follows that
\begin{eqnarray*}
A&=&
\int_\mathbb R (\tpsii  H \tpsi f)^2 \Re \left(\frac 1{\phip'} \right) dx- \int_\mathbb R (S f)^2 \Re \left(\frac 1{\phip'} \right) dx 
\\
&& -\mu\int_\R (H \tpsi f)^2 \Re \left(\frac 1{\phim'} \right) dx +  \mu \int_\mathbb R (\tpsi f)^2 \Re \left(\frac 1{\phim'} \right) dx
\\
&=&
\int_\mathbb R (\tpsii  H \tpsi f)^2 \Re \left(\frac 1{\phip'} \right) dx- \int_\mathbb R (S f)^2 \Re \left(\frac 1{\phip'} \right) dx
\\
& &- \mu\int_\R (\tpsii H \tpsi f)^2 \Re \left(\frac 1{\phip'} \right) dx + \mu \int_\mathbb R f^2 \Re \left(\frac 1{\phip'} \right) dx\\
&=&
(1-\mu) \int_\R (\tpsii H \tpsi  f)^2\Re \left(\frac 1{\phip'} \right) dx 
\\
& &-
 \int_\R (S f)^2 \Re \left(\frac 1{\phip'} \right) dx +\mu \int_\mathbb R f^2 \Re \left(\frac 1{\phip'} \right) dx
\\
&=&
(1-\mu) \int_\R (\tpsii H \tpsi f)^2 \Re \left(\frac 1{\phip'} \right) dx+ A_1+A_2.
\end{eqnarray*}
We have
\begin{eqnarray*}
A_1&=& -  \frac 1{(1-\mu)^2}  \int_\R ( (S-\mu I)f -\mu ((S-I)f)  )^2 \Re \left(\frac 1{\phip'} \right) dx
\\
&=&
 -  \frac 1{(1-\mu)^2}\left[\int_\R ((S-\mu I)f)^2 \Re \left(\frac 1{\phip'} \right) dx+ \mu^2 \int_\R ( (S-I)f)^2 \Re \left(\frac 1{\phip'} \right) dx \right]
 \\
&+&  \frac{2\mu}{(1-\mu)^2} \int_\R ((S-\mu I)f) ((S-I)f) \Re \left(\frac 1{\phip'} \right) dx,
\end{eqnarray*}
and 
\begin{eqnarray*}
A_2&=&\frac{\mu}{(1-\mu)^2} \int_\R ((S-\mu I)f- (S-I)f)^2 \Re \left(\frac 1{\phip'} \right) dx 
\\
&=&
\frac{\mu}{(1-\mu)^2}\left[ \int_\R ((S-\mu I)f)^2 \Re \left(\frac 1{\phip'} \right) dx+ \int_\mathbb R ( (S-I)f)^2 \Re \left(\frac 1{\phip'} \right) dx \right]
\\
&-&  \frac{2\mu}{(1-\mu)^2} \int_\R ((S-\mu I)f) ((S-I)f) \Re \left(\frac 1{\phip'} \right) dx.
\end{eqnarray*}
Therefore,
\begin{equation*}
A_1+A_2= -\frac{1}{1-\mu}\int_\R ((S-\mu I)f)^2 \Re \left(\frac 1{\phip'} \right) dx 
 +\frac{\mu}{1-\mu}\int_\R ( (S-I)f)^2 \Re \left(\frac 1{\phip'} \right) dx.
\end{equation*}
Putting all together, we obtain
\begin{align*}
2\int_\mathbb R (S-\mu I)f \,(\tpsii H\tpsi  f) \Im \left(\frac 1{\phip'} \right)dx
=&(1-\mu) \int_\R (\tpsii H \tpsi f)^2 \Re \left(\frac 1{\phip'} \right) dx\\
&-\frac{1}{1-\mu}\int_\R ((S-\mu I)f)^2 \Re \left(\frac 1{\phip'} \right) dx \\
 &+\frac{\mu}{1-\mu}\int_\R ( (S-I)f)^2 \Re \left(\frac 1{\phip'} \right) dx.
\end{align*}
Equivalently, 
\begin{align}
&(1-\mu) \| \tpsii H \tpsi f\|_{\lpr{2}{\Re (1/{\phip'})}}^2
 +\frac{\mu}{1-\mu}\| (S-I)f\|_{ \lpr{2}{\Re(1/{\phip'})}}^2\nonumber\\
& \quad \quad=
2\int_\R (S-\mu I)f \,(\tpsii H\tpsi  f) \Im \left(\frac 1{\phip'} \right)dx +\frac{1}{1-\mu}\|(S-\mu I)f\|_{\lpr{2}{\Re (1/{\phip'})}}^2.\label{eq:temp}
\end{align}
If  $\varepsilon>0$,  we have 
\begin{align}
2\int_\R (S-\mu I)f &\,(\tpsii H\tpsi  f) \Im \left(\frac 1{\phip'} \right)dx\nonumber \\
&\lesssim \frac{1}{\varepsilon}\|(S-\mu I)f\|_{\lpr{2}{|\phip'|^{-1}}}^2+\varepsilon \|\tpsii H\tpsi  f\|_{\lpr{2}{|\phip'|^{-1}}}^2. \label{eq:temp2}
\end{align}
Recalling   that $\Re(\frac{1}{\phip'})\approx |\phip'|^{-1}$ since $\Lambda$ is a Lipschitz curve,  \eqref{eq:temp} and \eqref{eq:temp2} with $\varepsilon$ sufficiently small  give that 
\begin{equation*}
\|f\|_{\lpr{2}{|\phip'|^{-1}}}\approx \|(S- I)f\|_{\lpr{2}{|\phip'|^{-1}}}\lesssim \|(S-\mu I)f\|_{\lpr{2}{|\phip'|^{-1}}},
\end{equation*}
where we have used the $S-I$ is invertible and bounded in $\lpr{2}{|\phip'|^{-1}}$ (see Remark~\ref{re:S-I}).  Therefore  $S-\mu I$ is injective and has closed range. An application of Lemma~\ref{lem:LipTmu} with $\mu_0=0$ implies that $S-\mu I$ is invertible.
\end{proof}

\section{Solvability of \eqref{eq:TP} in weighted $L^2$ spaces}\label{sec:solTP}

As explained in Section~\ref{sec:intro}, the transmission problem~\eqref{eq:TP}  is associated to $P_\Psi(\mu)$ with $\Psi=\phipi\circ \phim$ and datum $f=\tphim(g)$,  where $\phipm:\hppm\to \opm$ are conformal maps onto upper and lower graph domains as described in Section~\ref{sec:phi} with $\Lambda=\partial\op=\partial\om.$ We note that $\Psi'>0$ almost everywhere  and we  assume that   $\Psi$  is locally absolutely continuous.

In this section we prove Theorem~\ref{thm:tplwgral}, which  gives necessary and sufficient conditions for the solvability of \eqref{eq:TP} in $\lpl{p}{\nu}$  in terms of solvability results  for $\tp{\Psi}{\mu}.$ We also present particular cases of Theorem~\ref{thm:tplwgral} using the results in Section~\ref{sec:solvL2} to obtain solvability of  \eqref{eq:TP} in $\lplnw{2}$, $\lpl{2}{|(\phipi)'|^{-1}}$ and $\lpl{2}{|(\phimi)'|^{-1}}.$

\begin{proof}[Proof of Theorem~\ref{thm:tplwgral}] We first show that $\tp{\Psi}{\mu}$ is solvable in $\lpr{p}{w}$ if and only if  the transmission problem~\eqref{eq:TP} is solvable in $\lpl{p}{\nu}$.

Assume first that $\tp{\Psi}{\mu}$ is solvable in $\lpr{p}{w}$ and let $g\in \lpl{p}{\nu}.$  Then \eqref{eq:normequiv} gives that $\tphim g\in \lpr{p}{w}$ and $\|g\|_{\lpl{p}{\nu}}=\|\tphim g\|_{\lpr{p}{w}}.$
Since $\tp{\Psi}{\mu}$ is solvable in $\lpr{p}{w}$, there exist solutions $\upm$ of $\tp{\Psi}{\mu}$ with datum $f=\tphim g.$ Define $\vpm=\upm\circ \phipmi$, which are harmonic in $\opm.$ By Item~\eqref{item:pmuc} in Definition~\ref{def:pmu}, we have
\begin{align}\label{eq:tpL2gral}
\|\ntmp_\alpha  \nabla\up\|_{\lpr{p}{\wtilp{p}}}\lesssim \|g\|_{\lpl{p}{\nu}} \quad \text{and}\quad
\|\ntmm_\alpha  \nabla\um\|_{\lpr{p}{w}}\lesssim \|g\|_{\lpl{p}{\nu}},
\end{align}
as required in Item~\eqref{item:TPsolc} of Definition~\ref{def:TPsol}.

Regarding Item~\eqref{item:TPsola}  of Definition~\ref{def:TPsol}, we have that $\vpm$  on $\Lambda$ are the traces of $\vpm$  in the sense of $\phipm$ non-tangential convergence since $\upm$  on $\R$ are the traces of $\upm$  in the sense of  non-tangential convergence by Item~\eqref{item:pmua}  of Definition~\ref{def:pmu}. Also, the proof of \cite[Theorem 1.4]{MR4542711} gives that for $\xi\in \Lambda$  such that $\phip'(\phip^{-1}(\xi))$ exists and is non-zero and  $z=x+i y$ with $y>0,$ 
\begin{align}\label{eq:cno}
\nabla \vp(z)\cdot {\bf n}(\xi)= |\phip'(\phip^{-1}(z))|^{-1}\text{Re}\left(\left(\fr{\partial \up}{\partial x}(\phip^{-1}(z))-i\fr{\partial \up}{\partial y}(\phip^{-1}(z))\right) \fr{|\phip'(\phip^{-1}(z))|}{\phip'(\phip^{-1}(z))} \,{\bf n}(\xi)\right),
\end{align}
where   on the left hand side we have the dot product of $\nabla \vp(z)$ with ${\bf n}(\xi)$ and,  on the right hand side,  ${\bf n}(\xi)$ is being multiplied as a complex number. A similar formula holds for $\vm$ with $\um,$ $\phim$ and $z=x+i y$ with $y<0.$ 
From here, noting that $\partial_y \up\in \lpr{p}{\wtilp{p}}$ and $\partial_y \um\in \lpr{p}{w}$ by \eqref{eq:tpL2gral}, it follows that 
$\partial_\normal\vpm$ on $\Lambda$ are the traces of $\partial_\normal\vpm$ in the sense of $\phipm$ non-tangential convergence and $\partial_y \upm=\tphipm(\partial_\normal \vpm).$  The later formula also shows that $\partial_{\bf{n}}\vpm\in \lpl{p}{\nu}$ by \eqref{eq:normequiv} (note that $\wtilp{p}=|\phip'|^{1-p}\,\nu\circ\phip$). 

Finally,   since $\up\circ\Psi=\um$ and  $\tpsi(\partial_y \up)-\mu\, \partial_y \um=f$ almost everywhere in $\R$ (by Item~\eqref{item:pmub}  of Definition~\ref{def:pmu}), we also have that
$\vp=\vm$ and $\partial_{\normal}\vp- \mu \,\partial_\normal \vm =g$ almost everywhere on $\Lambda$ with respect to arc length, as seen in Section~\ref{sec:intro}; therefore Item~\eqref{item:TPsolb}  of Definition~\ref{def:TPsol} is satisfied.

Conversely, assume that the transmission problem~\eqref{eq:TP} is solvable in $\lpl{2}{\nu}$ and let $f\in \lpr{2}{w}.$ Setting $g=T_{\phimi}f,$ which satisfies $\|g\|_{\lpl{2}{\nu}}=\|f\|_{\lpr{2}{w}}$ by \eqref{eq:normequiv}, let $\vpm$ be the solutions of \eqref{eq:TP} with datum $g.$ 
It then follows that  $\upm=\vpm\circ\phipm,$ which are harmonic in $\hppm,$ solve $\tp{\Psi}{\mu}$ with datum $f:$  

Item~\eqref{item:pmuc} in Definition~\ref{def:pmu} follows from Item~\eqref{item:TPsolc} of Definition~\ref{def:TPsol}. 

Regarding Item~\eqref{item:pmua} in Definition~\ref{def:pmu}, we have that $\upm$ on $\R$ are the traces of $\upm$ in the sense of non-tangential convergence since $\vpm$ on $\Lambda$ are the traces of $\vpm$ in the sense of $\phipm$ non-tangential convergence. Also, for $(x,y)\in \hpp,$ we have
\begin{align*}
\partial_y \up(x,y)=(\partial_1\vp\circ\phip)(x,y)\,\partial_y(\Re(\phip))(x,y)+(\partial_2\vp\circ\phip)(x,y)\,\partial_y(\Im(\phip))(x,y)
\end{align*}
and similarly for $\um$ with   $\vm,$ $\phim$ and $(x,y)\in \hpm.$ Then, since $\partial_{\bf{n}}\vpm$ converge $\phipm$ non-tangentially and $\phipm'$ converge non-tangentially, we obtain that $\partial_y \upm$ exists on $\R$ as the trace of  $\partial_y \upm$ in the sense of non-tangential convergence. 

As for Item~\eqref{item:pmub}  of Definition~\ref{def:pmu}, note first that an analogous formula to the above for $\partial_x \up$ allows to conclude that $\partial_x \upm$ exists on $\R$ as the trace of  $\partial_x \upm$ in the sense of non-tangential convergence.  We can then use \eqref{eq:cno} and its counterpart for $\vm$ and $\um$ to conclude that $\partial_y \upm=\tphipm(\partial_\normal \vpm)$ almost everywhere and proceed as in Section~\ref{sec:intro} to obtain that $\tpsi(\partial_y \up)-\mu\, \partial_y \um=f$ almost everywhere in $\R$. Finally,   $\up\circ\Psi=\um$ almost everywhere  in $\R$  since $\vp=\vm$ almost everywhere  in $\Lambda.$

The equivalence for unique solvability follows from the relationship $\upm=\vpm\circ\phipm$ between the solutions of $\tp{\Psi}{\mu}$ and \eqref{eq:TP}. 
\end{proof}

\subsubsection{Particular cases of Theorem~\ref{thm:tplwgral}} In this section, we state results on the solvability of  \eqref{eq:TP} in $\lplnw{2}$, $\lpl{2}{|(\phipi)'|^{-1}}$ and $\lpl{2}{|(\phimi)'|^{-1}}.$

\begin{enumerate}[(a)]

\item {\bf{Solvability in $\lplnw{2}$}}: This corresponds to $\nu=1$,  $w=|\phim'|^{-1}$ and $\wtil=|\phip'|^{-1}$.  The estimates for $\vpm$ then become  
\begin{align*}
\|\ntmp_\alpha  \nabla(\vp\circ \phip)\|_{\lpr{2}{|\phip'|^{-1}}}\lesssim \|g\|_{\lplnw{2}}  \quad \text{and}\quad 
\|\ntmm_\alpha  \nabla(\vm\circ\phim)\|_{\lpr{2}{|\phim'|^{-1}}}\lesssim \|g\|_{\lplnw{2}}.
\end{align*}

We note that in the case that $\opm$ are upper and lower Lipschitz graph domains, these estimates  can be rewritten in the following form:
\begin{align*}
\|\widetilde{\mathcal{M}}^{+}_\alpha  \nabla\vp\|_{\lplnw{2}}\lesssim \|g\|_{\lplnw{2}} \quad \text{and}\quad
\|\widetilde{\mathcal{M}}^{-}_\alpha  \nabla\vm\|_{\lplnw{2}}\lesssim \|g\|_{\lplnw{2}},
\end{align*}
where $\widetilde{\mathcal{M}}^{\pm}_\alpha$ are the non-tangential maximal operators associated to the upper and lower Lipschitz graph domains, respectively. This follows from the facts 
\begin{align}
\|\ntmp_\alpha  \nabla(\vp\circ\phip)\|_{\lpr{2}{|\phip'|^{-1}}}&\approx \|\widetilde{\mathcal{M}}^{+}_\alpha \nabla\vp\|_{\lplnw{2}},\label{eq:maxhardy1}\\
\|\ntmm_\alpha  \nabla(\vm\circ\phim)\|_{\lpr{2}{|\phim'|^{-1}}} &\approx \|\widetilde{\mathcal{M}}^{-}_\alpha  \nabla\vm\|_{\lplnw{2}}. \label{eq:maxhardy2}
\end{align}
For \eqref{eq:maxhardy1} see the proof of  \cite[Theorem 1.4]{MR4542711}. The equivalence \eqref{eq:maxhardy2} is a consequence of  the latter as follows:  Let $\otil=-\om$,  $v(z)=\vm(-z)$ for $z\in \otil$ and $\Phi:\hpp:\to\otil$ be defined by $\Phi(x,y)=-\phim(-x,-y).$ It follows that 
\begin{equation*}
\ntmm_\alpha\nabla(\vm\circ\phim)(-x)=\ntmp_\alpha\nabla(v\circ\Phi)(x),\quad x\in \R,
\end{equation*}
which leads to 
\begin{equation}\label{eq:maxhardy3}
\|\ntmm_\alpha\nabla(\vm\circ\phim)\|_{\lpr{2}{|\phim'|^{-1}}}=\|\ntmp_\alpha\nabla(v\circ\Phi)\|_{\lpr{2}{|\Phi'|^{-1}}}.
\end{equation}
The equivalence \eqref{eq:maxhardy1} applied with $v$, $\Phi$ and $\Ltil=-\Lambda$ and the fact that $\widetilde{\mathcal{M}}^+_\alpha\nabla v(z)=\widetilde{\mathcal{M}}^-_\alpha\nabla \vm (-z)$ give 
\begin{equation}\label{eq:maxhardy4}
\|\ntmp_\alpha\nabla(v\circ\Phi)\|_{\lpr{2}{|\Phi'|^{-1}}}\approx \|\widetilde{\mathcal{M}}^{+}_\alpha \nabla v\|_{L^2(\Ltil)}=\|\widetilde{\mathcal{M}}^{-}_\alpha \nabla \vm\|_{\lplnw{2}}.
\end{equation}
Using \eqref{eq:maxhardy3} and \eqref{eq:maxhardy4}, we obtain \eqref{eq:maxhardy2}.

\bigskip

The following corollary of Theorems~\ref{thm:welding} and \ref{thm:tplwgral}  recovers results in \cite[Theorem 1.1]{MR2091359} for $p=2$ and $n=2$ for upper and lower Lipschitz graph domains.  

\begin{corollary}\label{coro:ezmi}  If  $\opm$ are upper and lower Lipschitz graph domains, the transmission problem \eqref{eq:TP} is uniquely solvable in $\lplnw{2}$ for every $\mu>0.$ 
\end{corollary}

We note that the transmission problem \eqref{eq:TP} studied in \cite{MR2091359} requires  the seemingly more general condition $\up-\um=h$ (rather than $\up-\um=0$) where  $h$ and its derivative are in $\lplnw{2}$. However, using \cite[Theorem 5.1]{MR769382}, Corollary~\ref{coro:ezmi} also holds for this inhomogeneous version of \eqref{eq:TP}.

\bigskip

\item {\bf{Solvability in $\lpl{2}{|(\phipi)'|^{-1}}$}}: This corresponds to $\nu=|(\phipi)'|^{-1}$,  $w=\frac{1}{\Psi'}$, and  $\wtil=1.$ The estimates for $\vpm$ are given by 
\begin{align*}
\|\ntmp_\alpha  \nabla(\vp\circ\phip)\|_{\lprnw{2}}&\lesssim \|g\|_{\lpl{2}{|(\phipi)'|^{-1}}}\\
\|\ntmm_\alpha  \nabla(\vm\circ\phim)\|_{\lpr{2}{\frac{1}{\Psi'}}}&\lesssim \|g\|_{\lpl{2}{|(\phipi)'|^{-1}}}.
\end{align*}

The following result follows from Theorems~\ref{thm:hypergral} and \ref{thm:tplwgral}:

\begin{corollary}\label{coro:psi2} Assume  $\Psi'\in \apr{2},$  $\apsi$ and $\kpsi$ are as given in Theorem~\ref{thm:hypergral} and $\mu_0$ is the positive root of $\apsi-\mu^2 - 2\,\kpsi\,\mu =0.$ Then  the transmission problem \eqref{eq:TP} is uniquely solvable in $\lpl{2}{|(\phipi)'|^{-1}}$ for $\mu\neq 0$ such that $|\mu|<\mu_0$ or $|\mu|>1/\mu_0$. 
\end{corollary}

Corresponding statements analogous to Corollary~\ref{coro:psi2} follow from  Theorems~\ref{thm:hyper} and \ref{thm:hilb}. 

\bigskip

\item {\bf{Solvability in $\lpl{2}{|(\phimi)'|^{-1}}$}}: This case corresponds to $\nu=|(\phimi)'|^{-1}$, $w=1$ and $\wtil=|(\Psi^{-1})'|^{-1}$. The estimates for $\vpm$ are then 
\begin{align*}
\|\ntmp_\alpha  \nabla(\vp\circ\phip)\|_{\lpr{2}{|(\Psi^{-1})'|^{-1}}}&\lesssim \|g\|_{\lpl{2}{|(\phimi)'|^{-1}}}\\
\|\ntmm_\alpha  \nabla(\vm\circ\phim)\|_{\lprnw{2}}&\lesssim \|g\|_{\lpl{2}{|(\phimi)'|^{-1}}}.
\end{align*}

Corollary~\ref{coro:hyper2}, whose proof is presented in Section~\ref{sec:hyperbola}, is a particular case of this setting.

\end{enumerate}

\section{The hyperbola}\label{sec:hyperbola}

In this section, we consider the Jordan curve given by the hyperbola $y=1/x,$ $x>0:$ we present an application of Theorem~\ref{thm:hypergral}  and we prove Corollary~\ref{coro:hyper2}.

\subsection{An application of Theorem~\ref{thm:hypergral}}\label{sec:hyper}
Let $\opm$ be the upper and lower graph domains associated to  $y=1/x$, $x>0.$  A conformal map from $\op$ onto $\hpp$ is given by $z^2-2i$; consider the inverse of this map:
$$
\Phi_+ (z)= (z+2i)^{1/2},\quad z\in \hpp.
$$
Also, if $h(z)= z^3-3z$, then $ -ih^{-1}(-iz^2 -2)$ is a conformal map from $\om$ onto $\hpm;$ consider the inverse of this map:
 $$
\Phi_- (z)= (z^3 +3z +2i)^{1/2} ,\quad z\in \hpm.
$$
We refer the reader to \cite{hyperbola} for the maps $\phip$ and $\phim$, where  we have checked that all statements and computations claimed  are correct.

 \begin{figure}[htbp]
\begin{center}
\begin{tikzpicture}[scale=.5]
\filldraw[draw=white,bottom color=lightgray, top color=white] (0, 0)-- (0,4)--(8,4)--(8,0) ;
\draw [->, thick] (0,0) -- (8,0);
\draw [->, thick] (4,-4) -- (4,4);
\node[below] at (6,3) {\large$\R^2_+$};
\node[above] at (6,-3) {\large$\R^2_-$};
\draw [->, thick] (16,-1) -- (25,-1);
\draw [->, thick] (18,-4) -- (18,4);
\filldraw[draw=white,bottom color=lightgray, top color=white] (18, 4) -- plot[smooth,samples=100, domain=18.2:25, thick] ({\x}, {0.9*1/(\x-18)-0.8}) -- (25,4) -- cycle;
\draw [smooth,samples=100, domain=18.2:25,  thick] plot({\x}, {0.9*1/(\x-18)-0.8});
\node [below] at (22, 2) {\large$\Omega^+$};
\node [above] at (20, -3) {\large$\Omega^-$};
\draw[thick,->] (10,3) arc (130:40:4);
\draw[thick,->] (10,-3) arc (230:320:4);
\node [above] at (13,4) {{$\Phi_+(z)= (z+2i)^\frac{1}{2}$}};
\node [below] at (13,-4) {{$\Phi_-(z)=(z^3+3z+2i)^\frac{1}{2}$}};
\end{tikzpicture}
\label{fig:tranportealplano}
\end{center}
\end{figure}
It follows that 
$$
\Psi(x)=\Phi_+^{-1}(\Phi_- (x) )= x^3+3x, \quad x\in \R,
$$
and $\Psi'(x)= 3(1+x^2)$, which does not belong to $\apr{2}.$ 
Hence  $\Psi$ does not satisfy the hypothesis of neither  Theorem \ref{thm:hypergral} nor Theorem~\ref{thm:hilb}. However, setting $\Theta=\Psi^{-1},$ we have
\begin{equation}\label{eq:hyperpsi}
\Theta(x)=\left(\frac{x+\sqrt{4+x^2}}2\right)^{1/3}- \left(\frac 2{x+\sqrt{4+x^2}}\right)^{1/3},
\end{equation}
and
$$
\frac 1{\Theta'(x)}= 3\cdot 2^{2/3} \frac{\sqrt{4+x^2}(x+ \sqrt{4+x^2})^{1/3}}{2+2^{1/3 }(x+ \sqrt{4+x^2})^{2/3}}  \approx (4+x^2)^{\frac{1}{3}}
\in \apr{2}.$$
We can then apply Theorem \ref{thm:hypergral} with $\Theta$ and obtain that if $0<|\mu|<1$ satisfies $A_\Theta-\mu^2 - 2\,k_\Theta\,|\mu|>0$, then $\tp{\Theta}{\mu}$ and $\tp{\Theta}{1/\mu}$ are  uniquely solvable in $\lpr{2}{\frac{1}{\Theta'}}$. Also, since $\wtil=1$ for $w= \frac{1}{\Theta'}$, Part~\eqref{coro:sym:b} of Corollary~\ref{coro:sym} gives that $\tp{\Psi}{\mu}$ and $\tp{\Psi}{1/\mu}$ are uniquely solvable in $\lprnw{2}$ for the same range of $\mu$.

We next give an explicit value of $\mu_0$ such that $\tp{\Theta}{\mu}$ is uniquely solvable for $|\mu|<\mu_0$ and $|\mu|>1/\mu_0$ by presenting an example of a conformal map $\Phi$ associated to $\Theta$ in \eqref{eq:hyperpsi} according to the statement of Theorem \ref{thm:hypergral}. Consider the conformal map 
\begin{equation*}
\Phi(z)= i (8-4zi )^{\frac{1}{3}},\quad z\in \hpp. 
\end{equation*}
We first need to compute $A_\Theta$ and $k_\Theta.$ We have 
\begin{equation*}
\Phi'(x)= \frac{4}{3} (8-4xi)^{-\frac{2}{3}},\quad x\in \R;
\end{equation*}
noting that $8-4xi=4\sqrt{4+x^2}e^{i\arctan(-\frac{x}{2})}$ for $x\in \R,$ we obtain
 \begin{equation*}
 \frac{1}{\Phi'(x)}=3\left(\frac{4+x^2}{4}\right)^{\frac{1}{3}}e^{i\frac{2}{3}\arctan(-\frac{x}{2})}, \quad x\in \R.
 \end{equation*}
 Then
 \begin{equation*}
 \Re\left(\frac{1}{\Phi'(x)}\right)=3\left(\frac{4+x^2}{4}\right)^{\frac{1}{3}}\cos\left(\frac{2}{3}\arctan\left(-\frac{x}{2}\right)\right)
 \end{equation*}
 and, since $-\frac{\pi}{3}\le \frac{2}{3}\arctan(-\frac{x}{2})\le \frac{\pi}{3},$ we get $\frac{1}{2}\le \cos(\frac{2}{3}\arctan(-\frac{x}{2}))\le 1.$ Therefore
 \begin{equation*}
\frac{3}{2}\left(\frac{4+x^2}{4}\right)^{\frac{1}{3}}\le   \Re\left(\frac{1}{\Phi'(x)}\right)\le 3\left(\frac{4+x^2}{4}\right)^{\frac{1}{3}}
 \end{equation*}
 and 
 \begin{equation*}
\frac{3}{2}\Theta'(x)\left(\frac{4+x^2}{4}\right)^{\frac{1}{3}}\le \Theta'(x)  \Re\left(\frac{1}{\Phi'(x)}\right)\le 3\Theta'(x)\left(\frac{4+x^2}{4}\right)^{\frac{1}{3}}.
 \end{equation*}
 It can be checked that
 \begin{equation}\label{eq:est1}
4^{-\frac{1}{3}}\le 3 \Theta'(x)\left(\frac{4+x^2}{4}\right)^{\frac{1}{3}}\le1;
 \end{equation}
then \eqref{eq:hypergral0}  holds with $A_\Theta=2^{-\frac{5}{3}}.$ Regarding $k_\Theta,$ we have
\begin{align*}
k_\Theta= \left\| \Theta' \im \right\|_{\lprnw{\infty}}&=\left\| 3\left(\frac{4+x^2}{4}\right)^{\frac{1}{3}}\Theta'(x)\sin\left(\frac{2}{3}\arctan\left(-\frac{x}{2}\right)\right)\right\|_{\lprnw{\infty}}\\&\le \sin\left(\frac{\pi}{3}\right)=\frac{\sqrt{3}}{2},
\end{align*}
where in the last inequality we have used \eqref{eq:est1} and the fact that $|\frac{2}{3}\arctan\left(-\frac{x}{2}\right)|\le \pi/3.$
We next find $\mu_0;$ we have
\begin{equation*}
A_\Theta-\mu^2 - 2\,k_\Theta\,|\mu|\ge 2^{-\frac{5}{3}}-\mu^2 - \sqrt{3}\,|\mu| >0,
\end{equation*}
and solving $2^{-\frac{5}{3}}-\mu^2 - \sqrt{3}\,|\mu|=0$ for a positive root we obtain 
$$\mu_0=\frac{-\sqrt{3}+\sqrt{3+2^{1/3}}}{2}\approx 0.165953;$$ 
note that  $\frac{1}{\mu_0}=\frac{2}{-\sqrt{3}+\sqrt{3+2^{1/3}}}\approx6.02579.$

\subsection{Proof of Corollary~\ref{coro:hyper2}}
 Let $\phipm$ be as in Section~\ref{sec:hyper} and $\Psi=\phipi\circ \phim.$  According to the computations in Section~\ref{sec:hyper},  $\tp{\Psi}{\mu}$ is uniquely solvable in $\lprnw{2}$ for $\mu\neq 0$
such that $|\mu|<\mu_0$ or  $|\mu|>1/\mu_0.$ By Theorem~\ref{thm:tplwgral}, we then conclude that  \eqref{eq:TP} is uniquely solvable in $\lpl{2}{|(\phimi)'|^{-1}}$ for the same range of $\mu.$
\qed

\bigskip

\noindent {\bf{Acknowledgements}:} We thank Jos\'e M. Arrieta for useful discussions regarding the proof of Theorem~\ref{thm:uniqueN}.

\bibliographystyle{abbrv}
\bibliography{CNSbiblio}

\begin{thebibliography}{10}

\bibitem{MR0200442}
L.~V. Ahlfors.
\newblock {\em Lectures on quasiconformal mappings}, volume No. 10 of {\em Van
  Nostrand Mathematical Studies}.
\newblock D. Van Nostrand Co., Inc., Toronto, Ont.-New York-London, 1966.
\newblock Manuscript prepared with the assistance of Clifford J. Earle, Jr.

\bibitem{MR0086869}
A.~Beurling and L.~Ahlfors.
\newblock The boundary correspondence under quasiconformal mappings.
\newblock {\em Acta Math.}, 96:125--142, 1956.

\bibitem{MR1081289}
C.~J. Bishop.
\newblock Conformal welding of rectifiable curves.
\newblock {\em Math. Scand.}, 67(1):61--72, 1990.

\bibitem{MR2373370}
C.~J. Bishop.
\newblock Conformal welding and {K}oebe's theorem.
\newblock {\em Ann. of Math. (2)}, 166(3):613--656, 2007.

\bibitem{MR2676605}
M.~Borsuk.
\newblock {\em Transmission problems for elliptic second-order equations in
  non-smooth domains}.
\newblock Frontiers in Mathematics. Birkh\"{a}user/Springer Basel AG, Basel,
  2010.

\bibitem{MR4228861}
L.~A. Caffarelli, M.~Soria-Carro, and P.~R. Stinga.
\newblock Regularity for {$C^{1,\alpha}$} interface transmission problems.
\newblock {\em Arch. Ration. Mech. Anal.}, 240(1):265--294, 2021.

\bibitem{MR0094579}
S.~Campanato.
\newblock Sul problema di {M}. {P}icone relativo all'equilibrio di un corpo
  elastico incastrato.
\newblock {\em Ricerche Mat.}, 6:125--149, 1957.

\bibitem{CLN2023}
M.~J. Carro, T.~Luque, and V.~Naibo.
\newblock The {Z}aremba problem in two-dimentional {L}ipschitz graph domains.
\newblock {\em Preprint}, 2023.

\bibitem{MR4542711}
M.~J. Carro, V.~Naibo, and C.~Ortiz-Caraballo.
\newblock The {N}eumann problem in graph {L}ipschitz domains in the plane.
\newblock {\em Math. Ann.}, 385(1-2):17--57, 2023.

\bibitem{MR4674966}
M.~J. Carro, V.~Naibo, and M.~Soria-Carro.
\newblock Rellich identities for the {H}ilbert transform.
\newblock {\em J. Funct. Anal.}, 286(4):Paper No. 110271, 22, 2024.

\bibitem{MR3800109}
M.~J. Carro and C.~Ortiz-Caraballo.
\newblock On the {D}irichlet problem on {L}orentz and {O}rlicz spaces with
  applications to {S}chwarz-{C}hristoffel domains.
\newblock {\em J. Differential Equations}, 265(5):2013--2033, 2018.

\bibitem{MR1092919}
L.~Escauriaza, E.~B. Fabes, and G.~Verchota.
\newblock On a regularity theorem for weak solutions to transmission problems
  with internal {L}ipschitz boundaries.
\newblock {\em Proc. Amer. Math. Soc.}, 115(4):1069--1076, 1992.

\bibitem{MR2091359}
L.~Escauriaza and M.~Mitrea.
\newblock Transmission problems and spectral theory for singular integral
  operators on {L}ipschitz domains.
\newblock {\em J. Funct. Anal.}, 216(1):141--171, 2004.

\bibitem{MR1149120}
L.~Escauriaza and J.~K. Seo.
\newblock Regularity properties of solutions to transmission problems.
\newblock {\em Trans. Amer. Math. Soc.}, 338(1):405--430, 1993.

\bibitem{MR3243734}
L.~Grafakos.
\newblock {\em Classical {F}ourier analysis}, volume 249 of {\em Graduate Texts
  in Mathematics}.
\newblock Springer, New York, third edition, 2014.

\bibitem{MR312139}
R.~Hunt, B.~Muckenhoupt, and R.~Wheeden.
\newblock Weighted norm inequalities for the conjugate function and {H}ilbert
  transform.
\newblock {\em Trans. Amer. Math. Soc.}, 176:227--251, 1973.

\bibitem{MR4480881}
C.~Karafyllia and D.~Ntalampekos.
\newblock Extension of boundary homeomorphisms to mappings of finite
  distortion.
\newblock {\em Proc. Lond. Math. Soc. (3)}, 125(3):488--510, 2022.

\bibitem{MR545265}
C.~Kenig.
\newblock Weighted {H}ardy spaces on {L}ipschitz domains.
\newblock In {\em Harmonic analysis in {E}uclidean spaces ({P}roc. {S}ympos.
  {P}ure {M}ath., {W}illiams {C}oll., {W}illiamstown, {M}ass., 1978), {P}art
  1}, Proc. Sympos. Pure Math., XXXV, Part, pages 263--274. Amer. Math. Soc.,
  Providence, R.I., 1979.

\bibitem{MR556889}
C.~Kenig.
\newblock Weighted {$H^{p}$} spaces on {L}ipschitz domains.
\newblock {\em Amer. J. Math.}, 102(1):129--163, 1980.

\bibitem{MR864372}
C.~Kenig.
\newblock Elliptic boundary value problems on {L}ipschitz domains.
\newblock In {\em Beijing lectures in harmonic analysis ({B}eijing, 1984)},
  volume 112 of {\em Ann. of Math. Stud.}, pages 131--183. Princeton Univ.
  Press, Princeton, NJ, 1986.

\bibitem{MR3356996}
D.~Kriventsov.
\newblock Regularity for a local-nonlocal transmission problem.
\newblock {\em Arch. Ration. Mech. Anal.}, 217(3):1103--1195, 2015.

\bibitem{MR1770682}
Y.~Y. Li and M.~Vogelius.
\newblock Gradient estimates for solutions to divergence form elliptic
  equations with discontinuous coefficients.
\newblock {\em Arch. Ration. Mech. Anal.}, 153(2):91--151, 2000.

\bibitem{MR0089978}
J.~L. Lions.
\newblock Contribution \`a un probl\`eme de {M}. {M}. {P}icone.
\newblock {\em Ann. Mat. Pura Appl. (4)}, 41:201--219, 1956.

\bibitem{MR3586566}
D.~Mitrea, I.~Mitrea, M.~Mitrea, and M.~Taylor.
\newblock {\em The {H}odge-{L}aplacian}, volume~64 of {\em De Gruyter Studies
  in Mathematics}.
\newblock De Gruyter, Berlin, 2016.
\newblock Boundary value problems on Riemannian manifolds.

\bibitem{MR2299477}
I.~Mitrea and K.~Ott.
\newblock Counterexamples to the well-posedness of {$L^p$} transmission
  boundary value problems for the {L}aplacian.
\newblock {\em Proc. Amer. Math. Soc.}, 135(7):2037--2043, 2007.

\bibitem{MR1224587}
M.~\'{O}~Searc\'{o}id.
\newblock The continuity of the semi-{F}redholm index.
\newblock {\em Irish Math. Soc. Bull.}, (29):13--18, 1992.

\bibitem{picone}
M.~Picone.
\newblock Sur un probl\`eme nouveau pour l'\'equation lin\'eaire aux
  d\'eriv\'ees partielles de la th\'eorie mathematique classique de
  l'\'elasticit\'e.
\newblock {\em Colloque sur les \'equations aux d\'eriv\'ees partielles},
  Bruxelles, May 1954.

\bibitem{hyperbola}
\relax Mathematics Stack~Exchange.
\newblock https://math.stackexchange.com/questions/625468.

\bibitem{MR0131063}
M.~Schechter.
\newblock A generalization of the problem of transmission.
\newblock {\em Ann. Scuola Norm. Sup. Pisa Cl. Sci. (3)}, 14:207--236, 1960.

\bibitem{MR4658348}
M.~Soria-Carro and P.~R. Stinga.
\newblock Regularity of viscosity solutions to fully nonlinear elliptic
  transmission problems.
\newblock {\em Adv. Math.}, 435:Paper No. 109353, 52, 2023.

\bibitem{MR0082607}
G.~Stampacchia.
\newblock Su un problema relativo alle equazioni di tipo ellittico del secondo
  ordine.
\newblock {\em Ricerche Mat.}, 5:3--24, 1956.

\bibitem{MR769382}
G.~Verchota.
\newblock Layer potentials and regularity for the {D}irichlet problem for
  {L}aplace's equation in {L}ipschitz domains.
\newblock {\em J. Funct. Anal.}, 59(3):572--611, 1984.

\bibitem{MR2041702}
S.~Zakeri.
\newblock David maps and {H}ausdorff dimension.
\newblock {\em Ann. Acad. Sci. Fenn. Math.}, 29(1):121--138, 2004.

\end{thebibliography}
  
 \end{document}